\documentclass[12pt,twoside]{amsart}
\usepackage{graphicx,mathrsfs,epstopdf}
\usepackage{amsmath,amsfonts,amssymb,amsthm}
\usepackage[colorlinks,citecolor=blue,urlcolor=blue]{hyperref}
\usepackage{algorithm} % has to be loaded AFTER hypperrf
\usepackage{algorithmicx}
\usepackage{algpseudocode}
\usepackage{framed} % for highlighting paragraph
\usepackage{blkarray}
\usepackage{pgfplots}
\usepackage{natbib}
\usepackage{enumerate}
\usepackage{subfigure}
\usepackage{tikz}
\usetikzlibrary{chains,matrix,scopes,fit,decorations,positioning,shapes,snakes,arrows,decorations.markings}
\newtheorem{thm}{Theorem}[section]
\newtheorem{lem}[thm]{Lemma}
\newtheorem{prop}[thm]{Proposition}
\newtheorem{cor}[thm]{Corollary}

\theoremstyle{definition}
\newtheorem{defn}[thm]{Definition}

\newtheorem{ex}[thm]{Example}

\theoremstyle{remark}
\newtheorem{rem}[thm]{Remark}

\DeclareMathOperator{\cov}{cov} 

\DeclareMathOperator{\TE}{TE} 

\newcommand{\E}{\mathbb{E}}

\newcommand{\N}{\mathbb{N}}

\newcommand{\R}{\mathbb{R}}

\newcommand{\T}{\mathbb{T}}

\newcommand{\indic}{{1\!\!1}}

\newcommand{\bs}{\boldsymbol}
\newcommand{\mtp}{{\rm MTP}_{2}}

\newcommand{\cip}{\mbox{\,$\perp\!\!\!\perp$\,}}
\newcommand{\indep}{\cip}

\newcommand{\im}{{\mathcal{IM}}}

\title[Dependence in elliptical partial correlation graphs]{Dependence in elliptical partial correlation graphs}
\author[]{David Rossell}
\address{Department of Economics and Business, Universitat Pompeu Fabra, Barcelona, Spain}
\email{david.rossell@upf.edu}
\thanks{}
\author[]{Piotr Zwiernik}
\address{Department of Economics and Business, Universitat Pompeu Fabra, Barcelona, Spain}
\email{piotr.zwiernik@upf.edu}
\thanks{DR and PZ were supported from the Spanish Government grants (RYC-2015-18544,RYC-2017-22544,PGC2018-101643-B-I00,EUR2020-112096), and Ayudas Fundaci\'on BBVA a Equipos de Investigaci\'on Cientifica 2017. DR was partially supported by NIH grant R01 CA158113-01 and Europa Excelencia EUR2020-112096. }
\date{\today}                                            % Activate to display a given date or no date

\begin{document}
%\maketitle
\begin{abstract}

The Gaussian model {equips} strong properties that facilitate studying and interpreting graphical models. Specifically it reduces conditional independence and the study of positive association to determining partial correlations and their signs. When Gaussianity does not hold partial correlation graphs are a useful relaxation of graphical models, but it is not clear what information they contain (besides the obvious lack of linear association).
We study elliptical and transelliptical distributions as middle-ground between the Gaussian and other families that are more flexible but either do not embed strong properties or do not lead to simple interpretation.
We characterize the meaning of zero partial correlations in the elliptical family {and transelliptical copula models} and show that it retains much of the dependence structure from the Gaussian case. Regarding positive dependence, we prove impossibility results to learn (trans)elliptical graphical models, including that an elliptical distribution that is multivariate totally positive of order two for all dimensions must be essentially Gaussian. We then show how to interpret positive partial correlations as a relaxation, and obtain important properties related to faithfulness and Simpson's paradox.
We illustrate the transelliptical model potential to study tail dependence in S\&P500 data, and of positivity to {improve} regularized inference.

%We study (trans)elliptical distributions are used as a relaxation of graphical models typically defined through conditional independences. Since elliptical distributions allow for conditional independences only in the Gaussian case, it is not clear what information is contained in the partial correlation graph (apart from the obvious lack of linear association). This is even more obscured for transelliptical distributions. Motivated by this problem, in this paper we study the dependence structure in elliptical and transelliptical distributions. We provide a characterization of vanishing partial correlations that goes far beyond what has been known in the literature. In our analysis of positive dependence structure, we prove a number of impossibility results for elliptical distributions. We show that in high-dimensions an elliptical distribution that is multivariate totally positive of order two must be essentially Gaussian. 
\end{abstract}

\maketitle
%\tableofcontents

%\input intro

\section{Introduction}

Several papers study graphical models for elliptical and transelliptical distributions in the standard \citep{finegold2009robust,vogel2011elliptical} and high-dimensional settings \citep{barber2018rocket,bilodeau2014graphical,liu2012transelliptical,zhao2014calibrated}. These models found applications in many fields, such as finance and biology \citep{behrouzi2019detecting,stephens2013unified,vinciotti2013robust}, and (implicitly) wherever Gaussian graphical models were used but the underlying distribution is likely to depart from normality, e.g. be heavy-tailed or asymmetric. In the elliptical setting the usual definition of graphical models mimics the Gaussian case --- the model is given by zeros in the inverse covariance, or equivalently, by vanishing partial correlations. Despite this being a reasonable relaxation, the {corresponding} partial correlation graph (PG) cannot be interpreted in terms of conditional independence, since outside of the normal case no elliptical distributions allow for conditional independence (c.f. Proposition~\ref{prop:nocondind}). It is therefore unclear what type of dependence information is embedded by the PG.

For general distributions partial correlations inform only about linear dependence. Missing edges in the PG must then be interpreted with great care and, in some cases, they can fail to capture interesting dependence information. For example, in an aircraft data set from \cite{bowman1993density}, we can model dependence between the speed of an airplane and its wingspan. Although the sample correlation is negligible, more flexible dependence tests reveal that the variables are strongly related; see e.g. \cite{szekely2009brownian}. The reason is that for very fast (military) airplanes there is a negative dependence between speed and wingspan, while this dependence is positive for regular aircrafts.

The main theme of this paper is that for (trans)elliptical distributions there is {significantly} more information in the partial correlation graph beyond presence/absence of linear dependence. We introduce definitions and notation to aid the exposition.
\begin{defn}
A random vector $X=(X_1,\ldots,X_d)$ has an elliptical distribution  if there exists $\mu\in \R^d$ and a positive semi-definite matrix $\Sigma$ such that the  characteristic function of $X$ is of the form $\bs t\mapsto \phi(\bs t^T \Sigma \bs t)\exp(i \mu^T\bs t)$ for some $\phi:[0,\infty)\to \R$. We write $X\sim E(\mu,\Sigma)$ making $\phi$ in this notation implicit.
\label{def:elliptical_distrib}
\end{defn}
Important examples include the multivariate normal, Laplace and multivariate t-distributions.
Elliptical graphical models have been extended to transelliptical distributions (also known as elliptical copulas or meta-elliptical distributions,  \cite{fang2002meta,liu2012transelliptical}).

\begin{defn}
A random vector $Y$ has a transelliptical distribution with parameters $(\mu,\Sigma)$ if $f(Y):=(f_1(Y_1),\ldots,f_d(Y_d))\sim E(\mu,\Sigma)$
%there exist \emph{strictly increasing} functions $f_i$ such that  
{for some fixed strictly increasing functions $f_1,\ldots,f_d$.}
We write $Y \sim \TE(\mu, \Sigma)$, making $f$ in this notation implicit.
\label{def:transelliptical_distrib}
\end{defn}
Here the additional challenge is that $f$ is unknown. An  elegant approach to learning partial correlations relies on directly estimating the correlation matrix of $f(Y)$ without actually learning $f$; see \cite{liu2012transelliptical,lindskog2003kendall}, and then proceed as in the elliptical case (Section~\ref{ssec:rankcorrel}).

Throughout {we assume that $\Sigma$ is positive definite and}  denote $K= \Sigma^{-1}$, the set of vertices by $V=\{1,\ldots,d\}$, by $X_{(i)}$ the $d-1$ vector obtained by removing $X_i$ from $X$, and by $X_{(ij)}$ {the $d-2$ vector obtained by} removing $(X_i,X_j)$ from $X$.
Given $I,J\subseteq V$ denote by $X_I$ and $\mu_I$ the subvectors of $X$ and $\mu$ with coordinates in $I$ and by $\Sigma_{IJ}$ the corresponding subblock of $\Sigma$ with rows in $I$ and columns in $J$. 
The partial correlation between $(X_i,X_j)$ is 
\begin{equation}\label{eq:partialcor}
\rho_{ij\cdot V\setminus \{i,j\}}\;=\;-\frac{K_{ij}}{\sqrt{K_{ii}K_{jj}}}\quad\qquad\mbox{for all }i,j\in V	
\end{equation}
and so $\rho_{ij\cdot V\setminus \{i,j\}}=0$ if and only if $K_{ij}=0$.
Finally, we denote that $(X_i,X_j)$ are independent by $X_i \indep X_j$.

The usual interpretation of PGs in elliptical distributions is that --- since the conditional expectation $\E(X_i|X_{(i)})$ is linear in $X_{(i)}$ and the conditional correlation is equal to the partial correlation --- the condition $\rho_{ij \cdot V \setminus \{i,j\}}=0$ implies that $(X_i,X_j)$ are conditionally uncorrelated. That is, zero partial correlation implies zero conditional correlation. As we show in Theorem~\ref{th:anyf} something \emph{much} stronger is true. It is possible to fully characterize the PG in elliptical distributions: 
\begin{quote}
	The partial correlation $\rho_{ij \cdot V \setminus \{i,j\}}=0$ if and only if $\cov(g(X_i),X_j|X_{(ij)})=0$ for \textbf{every} function $g$ for which the covariance exists. 
\end{quote}
%To see why this may be important consider an extreme illustration when $X_1 \sim N(0,1)$ and $X_2=X_1^2$. Then $\rho_{12}=0$ but taking $f(x)=x^2$ we get ${\rm corr}(f(X_1),X_2)=1$. Theorem~\ref{th:anyf} shows that such an example is not possible if $(X_1,X_2)$ has an elliptical distribution.

A similar characterization extends to transelliptical distributions. The usual interpretation of $\rho_{ij \cdot V\setminus \{i,j\}}=0$ is that $f_i(Y_i), f_j(Y_j)$ are conditionally uncorrelated given $f_{(ij)}(Y)$, which is not very interesting, since $f$ is unknown. We show in Theorem~\ref{th:somef} that equivalently $\cov(f_i(Y_i),g(Y_j)|Y_{(ij)})=0$ for any $g$, provided the covariance exists. In particular, $\cov(f_i(Y_i),Y_j|Y_{(ij)})=0$, a more explicit dependence information in terms of $Y$.
{We also show in Proposition \ref{prop:partialcorr_condKendall} yet another equivalent interpretation, namely that $\rho_{ij \cdot V\setminus \{i,j\}}=0$ if and only if conditional Kendall's tau correlation between $g(Y_i)$ and $h(Y_j)$ being zero for all strictly increasing $g,h$.}

These findings are practically relevant. Recall that two variables $X_i$ and $X_j$ with general distribution are independent if and only if for all $L^2(\R)$ functions $g,h$ we have $\cov(g(X_i),h(X_j))=0$; see e.g. \cite[page 136]{feller1971ipt}.
That is, $X_i\indep X_j$ if and only if there is no way to transform $X_i$ and $X_j$ such that the new variables are correlated. Our characterization of $\rho_{ij \cdot V \setminus \{i,j\}}=0$ has an analogous interpretation, in elliptical families there is no way to transform $X_i$ such that the new variable is correlated with $X_j$.
{This rules out situations like the aircraft example above where speed is correlated with a non-linear function of wingspan.}

Further, interpreting $\rho_{ij \cdot V \setminus \{i,j\}}=0$ can be important in applications,
{given that elliptical and transelliptical models are a popular tool}
%. In particular (trans)elliptical models are often used 
to capture second-order or tail dependencies. Even though $\rho_{ij \cdot V \setminus \{i,j\}}=0$ such dependencies can be practically significant. As an example, let $X \sim E(0, \Sigma)$ and consider $\theta_{ij}={\rm corr}(X_i^2, X_j^2)$ as a simple measure of marginal tail (or second-order) dependence. {In  the  special sub-family of elliptical distributions defined by scale mixtures of normals}
%{that is, when $X$ satisfies the stochastic representation} $X= \mu + \tau^{-1/2} Z$ with $Z \sim N(0,\Sigma)$, $\tau>0$, {and $Z\indep \tau$} 
(c.f. Section~\ref{sec:review}), it is possible to show
\begin{align}
	\theta_{ij}={\rm corr}(X_i^2,X_j^2)\;=\; \lambda+(1-\lambda) \rho_{ij}^2,
\qquad\;\;\; \lambda\;=\;\frac{{\rm var}(\tfrac{1}{\tau})}{{\rm var}(\tfrac{1}{\tau})+2\E(\tfrac{1}{\tau^2})}\in [0,1],
\label{eq:corrsq}
\end{align}
where $\lambda=0$ if and only if $X$ is Gaussian. This measure is minimized for $\rho_{ij}=0$, then ${\rm corr}(X_i^2,X_j^2)=\lambda$, which can be non-negligible, e.g. $\lambda=1/(k-1)$ for the t-distribution with $k>4$ degrees of freedom. Figure \ref{fig:dependence_vs_partialcorr} shows this quantity and, for comparison, also the normalized mutual information (a standard measure of deviation from independence). Both measures converge to zero as $k\to\infty$ but this convergence is slow. Similarly, one may measure conditional tail dependence via ${\theta}_{ij \cdot V\setminus \{i,j\}}=$
\begin{align}
 {\rm corr}\Big((X_i - \E(X_i \mid X_{(ij)}))^2, (X_j - \E(X_j \mid X_{(ij)}))^2 \Big| X_{(ij)}\Big)\;=\; 
\lambda+(1-\lambda) \rho_{ij \cdot V \setminus \{i,j\} }^2,
\label{eq:corrsq_partial}
\end{align}
where the right-hand side follows from Proposition \ref{prop:margcond} below. When $\rho_{ij \cdot V\setminus \{i,j\}}=0$ the conditional tail dependence is $\lambda$. See Section~\ref{sec:examples} for further discussion and an illustration on stock market data.

\begin{figure}
\begin{center}
\includegraphics[scale=.7]{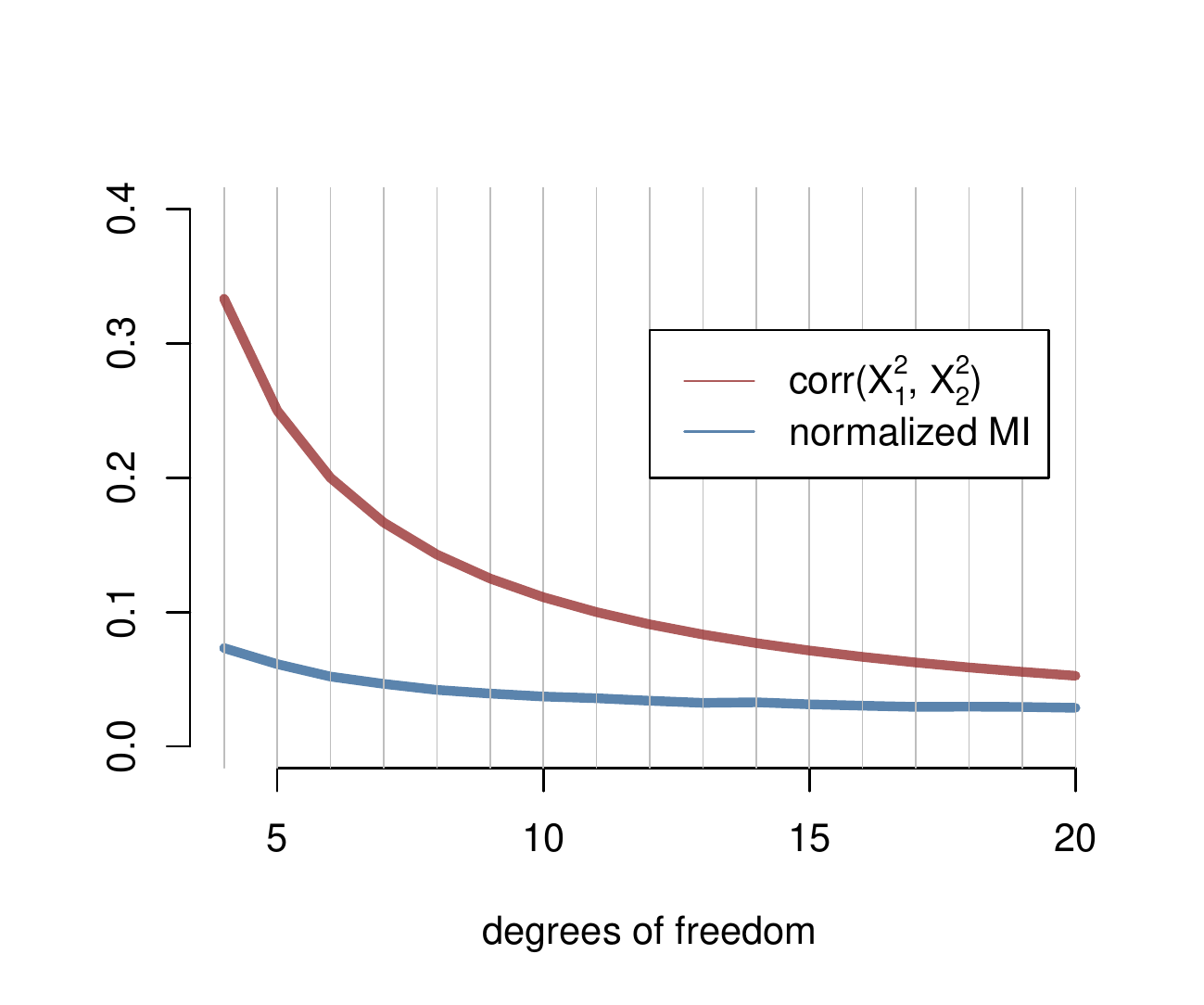}
\end{center}
\caption{${\rm corr} (X_i^2,X_j^2)$ for a multivariate t-distribution with $k$ degrees of freedom in the $\rho_{ij}=0$ case}
\label{fig:dependence_vs_partialcorr}
\end{figure}

Our other main contributions relate to PGs in settings where one wishes to study positive forms of association. Two standard ways to define positive dependence are via the notions of {\it multivariate total positivity of order two} ($\mtp$) and {\it conditionally increasing} (CI, Section~\ref{sec:TP}). Although these concepts are different, in the Gaussian case they are equivalent and reduce to constraining partial correlations to be non-negative. It is less clear how to interpret these concepts in general elliptical families. A first contribution is showing several impossibility results: within the elliptical family with at least one partial correlation zero there exist no conditionally increasing distributions (other than the Normal) implying the same result for $\mtp$ distributions. That is, if one wants remove edges in the PG with an additional positive dependence structure, one cannot rely on the standard notions of positive dependence. 

A natural relaxation is to learn a PG under the constraint that $\rho_{ij \cdot V \setminus \{i,j\}} \geq 0$, as proposed by \cite{agrawal2019covariance}. We refer to this strategy as {\it positive partial correlation graphs} (PPG). We contribute to understanding how should one interpret missing edges in the PPG, and to characterizing embedded positivity properties such as the positive correlation of each $X_i$ with any increasing function of the vector $X$. In Section~\ref{sec:examples} we illustrate how positivity constraints induce a type of regularization that can help improve inference relative to other standard forms of regularization, such as graphical LASSO, specifically attaining a higher log-likelihood with a sparser graph. This is meant as a testimony that our theoretical results have practical relevance. For further examples in risk modelling see \cite{abdous2005dependence,ruschendorf2017var}, and in psychology see \cite{epskamp2018tutorial,lauritzen2019total}, for example.

This paper also contributes to recent research aimed at understanding multivariate total positivity in a wide variety of contexts; see, for example, \cite{fallat:etal:17,lauritzen2019maximum,lauritzen2019total,robeva2018maximum,slawski2015estimation}. We provide in Theorem~\ref{th:abdous} a complete characterization of elliptical $\mtp$ distributions in terms of their density generator. In Theorem~\ref{thm:necessary_mtp2} this characterisation is used to show a remarkable result: a density generator may induce a $d$-variate $\mtp$ distribution for each $d\geq 2$ if and only if the underlying distribution is essentially Gaussian.

The paper is organized as follows. In Section~\ref{sec:review} we review basic results on elliptical distributions. In Section~\ref{sec:graphical} we characterize partial correlation graphs for elliptical and transelliptical distributions, giving a refined understanding of their encoded dependence information, {and how one should interpret zero partial correlations}. In Section~\ref{sec:TP} we study positive elliptical distributions and their alternative characterizations. In Section~\ref{sec:examples} we illustrate our main results with examples.

{In the rest of the article we employ the convention where the letter $X$ is reserved for elliptical random variables, $Y$ is reserved for transelliptical random variables, $Z$ is reserved for Gaussian random variables and $W$ for general random variables. The letter $f$ is reserved for the monotone functions defining transelliptical distributions as in Definition \ref{def:transelliptical_distrib}, $g$ and $h$ denote general functions, and $p$ denotes probability density functions.}

\section{Elliptical distributions}\label{sec:review}

We review some elliptical distribution results; for more information see  \cite{fang2018symmetric} or \cite{kelker1970distribution}, for example.

\subsection{Stochastic representation}

If $X\sim E(\mu,\Sigma)$ then $X$ admits the representation
\begin{equation}\label{eq:stoch}
	X\;=\;\mu+ \xi\,\cdot \Sigma^{1/2}\cdot U,
\end{equation}
where $\Sigma^{1/2}$ denotes the square root of {the  positive-definite} $\Sigma$, $\xi {\in \mathbb{R}^+}$ is {an arbitrary} random variable, $U \in \R^d$ is uniformly distributed on the unit $(d-1)$-dimensional sphere, and $\xi\indep U$. 
Elementary arguments show that, if $\E\xi<\infty$, then $\E X=\mu$ and {if $\E\xi^2<\infty$ then} ${\rm cov}(X)\;=\;\tfrac{\E(\xi^2)}{d}\cdot \Sigma$.  {Throughout this paper we assume that $0<\E \xi^2<\infty$. }

There is a useful representation equivalent to (\ref{eq:stoch}) in terms of normal variables.
Let $D^2 \sim \chi_d^2$ {with  $D^2\indep U$}. Then $Z= \sqrt{D^2}\, \Sigma^{1/2}\, U$ {is a mean  zero Gaussian variable with covariance $\Sigma$}. From (\ref{eq:stoch}) it follows that
\begin{align}\label{eq:stoch2}
  X=  \mu + \frac{1}{\sqrt{\tau}} \cdot Z,
\end{align}
where $\tau= D^2/\xi^2$ and $Z \sim N(0,\Sigma)$. {In general $\tau$ is not independent of $Z$.} The {special} case $\tau \indep Z$ corresponds to the \emph{scale mixture of normals} sub-family, which includes most popular elliptical distributions. The normal distribution corresponds to $\tau\equiv 1$. If $\tau \sim \chi^2_k/k$ for $k>2$ then $X$ has a multivariate t-distribution {with parameter $k$}. If $k=1$ we get the multivariate Cauchy, and if $\tau\sim {\rm Exp}(1)$ the multivariate Laplace distribution. 
{Equivalently, scale mixture of normals can be defined as the marginal distribution of $X$ associated to $(X,\tau)$ when $X \mid \tau \sim N(\mu,\tau^{-1/2} \Sigma)$, and $\tau \in \mathbb{R}^+$ follows an arbitrary continuous distribution.}

The elliptical family is closed under taking margins and under conditioning.

\begin{prop}\label{prop:margcond}
Let $X=(X_I,X_J) \sim E(\mu,\Sigma)$ be any split of $X$ into subvectors $X_I$ and $X_J$. Then
\begin{enumerate}
	\item [(i)] $X_I\sim E(\mu_I,\Sigma_{II})$,
	\item [(ii)] $X_I \mid X_J=x_J \sim E(\mu_{I|J},\Sigma_{II\cdot J})$, where $\mu_{I|J}:=\mu_I+\Sigma_{IJ}\Sigma_{JJ}^{-1}(x_J-\mu_J)$ and $$\Sigma_{II\cdot J}\;:=\;\Sigma_{II}-\Sigma_{IJ}\Sigma_{JJ}^{-1}\Sigma_{JI}\;=\; K_{II}^{-1}.$$
\end{enumerate}
\end{prop} 
For the proof see, for example, \cite[Theorem~2.18]{fang2018symmetric}. The  conditional mean $\mu_{I|J}$ has the same form as in the Gaussian case, where $\Sigma$ above can be replaced by ${\rm cov}(X)$ (a scalar multiple of $\Sigma$). Moreover, the conditional correlations $\mbox{corr}(X_i,X_j \mid X_{(ij)})$ are the normalized entries of $K=\Sigma^{-1}$ (the partial correlations $\rho_{ij \cdot V \setminus \{i,j\}}$), and do not depend on the value of the conditioning variable $X_{(ij)}$; see also Lemma~\ref{lem:condcov} below.

\subsection{Characterization of Gaussianity within the elliptical family}

If $X$ is Gaussian then each marginal distribution and each conditional distribution is Gaussian. Moreover, the conditional covariances do not depend on the conditioning variable and independence is equivalent to zero correlations. These properties characterize the Gaussian distribution in the class of elliptical distributions. We recall these basic results.   
\begin{lem}[Lemma~4 and 8 in \cite{kelker1970distribution}]\label{lem:kelker4}
	Let $X\sim E(\mu,\Sigma)$. If $X_I$ is Gaussian for some $I \subseteq V$ then $X$ is Gaussian. Further, if $X_I$ given $X_J$ is Gaussian for some $I,J \subseteq V$ then $X$ is Gaussian.
      \end{lem}
      
      We noted earlier that conditional correlations do not depend on the conditioning variable. For conditional covariances this is only true in the Gaussian case. 
\begin{lem}[Theorem~7 in \cite{kelker1970distribution}]\label{lem:condcov}
	Let $X=(X_I,X_J)\sim E(\mu,\Sigma)$. The conditional covariance of $X_I$ given $X_J$ is independent of $X_J$ if and only if $X$ is Gaussian.
\end{lem}

The standard definition of graphical models uses density factorizations that link to conditional independence through the Hammersley-Clifford theorem \citep{lau96}. However, it is not possible to define conditional independence in the elliptical family outside of the Gaussian case. The next two characterizations are the most consequential for this article.

\begin{lem}[Lemma~5 in \cite{kelker1970distribution}]\label{lem:kelker5}
	Let $X\sim E(\mu,\Sigma)$. If $\Sigma$ is a diagonal matrix, then the components of $X$ are independent if and only if $X$ has a normal distribution.
\end{lem}

\begin{prop}[Theorem~3 in \cite{baba2004partial}]\label{prop:nocondind}
	Suppose that $X\sim E(\mu,\Sigma)$ and $X_i\indep X_j|X_C$ for some $i,j\in V$ and $C\subseteq \{1,\ldots,d\}\setminus \{i,j\}$. Then $X$ is Gaussian.
\end{prop}
{In the special case of the scale mixture of normals, }the fact that Gaussianity is needed to conclude independence {in} Lemma~\ref{lem:kelker5} or conditional independence in Proposition~\ref{prop:nocondind} is also seen from (\ref{eq:corrsq}) and (\ref{eq:corrsq_partial}) respectively. For example, if $\tau$ is not constant then ${\rm corr}(X_i^2,X_j^2)>0$  proving that $(X_i,X_j)$ cannot be independent. 
%This gives an alternative way of proving these two results {for scale mixture of normals}.

\section{Graphs for (trans)elliptical distributions}\label{sec:graphical}

\subsection{Partial correlation graph and dependence}

By Proposition \ref{prop:nocondind} it is not possible to do structural learning in (non-normal) elliptical graphical models, under the conditional independence definition. It is then natural to look for relaxations that may be useful from the modelling point of view. A common strategy is to {define graphs based on zeroes} in the inverse covariance matrix, mimicking the Gaussian case; see \cite{vogel2011elliptical}. 

\begin{defn}
	The partial correlation graph (PG) is the graph $G=G(K)$ over vertex set $V=\{1,\ldots,d\}$ with an edge between $i\neq j$ if and only if $K_{ij}\neq 0$.
\end{defn}

Equivalently, $K_{ij}=0$ if and only if the {partial correlation} $\rho_{ij \cdot V \setminus \{i,j\}}=0$.
In general, $\rho_{ij \cdot V \setminus \{i,j\}}=0$ does not imply conditional   independence but only linear independence. The aim of this section is to understand what additional information does the PG carry in elliptical distributions. Proposition \ref{prop:margcond} and standard matrix algebra give
\begin{equation}\label{eq:meaninK}
\E(X_i|X_{(i)})\;= %\;\mu_i-\Sigma_{i,-i}\Sigma_{-i,-i}^{-1}(X_{(i)}-\mu_{(i)})\;=\;
\mu_i-\tfrac{1}{K_{ii}}K_{i,(i)}(X_{(i)}-\mu_{(i)}),	
\end{equation}
hence $K_{ij}=0$ if and only if $\E(X_i|X_{(i)})$ does not depend on $X_j$. This immediately gives the following standard result.
\begin{prop}\label{prop:condcorr}
Let $X\sim E(\mu,K^{-1})$. Then $K_{ij}=0$ if and only if $\cov(X_i,X_j|X_{(ij)})=0$, or equivalently, $\E(X_i|X_{(i)})$ and $\E(X_j|X_{(j)})$ depend on $X_{(ij)}$ only.
\end{prop}

%We now show that Proposition~\ref{prop:condcorr} gives a much stronger interpretation for missing edges in elliptical PGs. Recall that for any two variables $X_1,X_2$ it holds that $X_1\indep X_2$ are independent if and only if for all $L^2(\R)$ functions $g,h$ we have $\cov(g(X_1),h(X_2))=0$. If $(X_1,X_2)\sim E(\mu,\Sigma)$ then this condition cannot be satisfied for every $g,h$ (unless $(X_1,X_2)$ is Gaussian) with (\ref{eq:corrsq}) being an important example. 

Our first main results offer a stronger characterization for elliptical distributions. Lemma~\ref{lem:anyf} relates to marginal covariances, and immediately gives Theorem~\ref{th:anyf} on conditional covariances.
\begin{lem}\label{lem:anyf}
  Let $X \sim E(\mu,\Sigma)$ and $I,J \subseteq V$. Then $\Sigma_{IJ}=0$ if and only if $\cov(g(X_I),X_J)=0$ for any function $g$ for which the covariance exists. 
  %EARLIER STATEMENT
  %If $(X_1,X_2)$ is elliptically distributed and the scale matrix $\Sigma$ is diagonal then $\cov(f(X_1),X_2)=\cov(X_1,g(X_2))=0$ for all functions $f,g$ for which these covariances exist. 
\end{lem}
%Note that $g$ in Lemma \ref{lem:anyf} can map to an arbitrary real vector space $\R^k$. 
\begin{proof}
Without loss of generality assume $J=\{j\}$. By  the law of  total expectation
$$
\cov(g(X_I),X_j)=\E_{X_I}\Big[(g(X_I)-\E g(X_I))\E(X_i-\mu_j|X_I)\Big],
$$
where the expectation  $E_{X_I}$ is computed with respect to the marginal distribution of $X_I$. The expression inside the  expectation  is almost  surely zero with respect to this distribution. Indeed, by Proposition~\ref{prop:margcond},
$$
\E(X_j-\mu_j|X_I)\;=\;\E(X_j|X_I)-\mu_j\;=\;\Sigma_{jI}\Sigma_{II}^{-1}(X_I-\mu_I),
$$
which is zero because $\Sigma_{jI}=0$.
\end{proof}

\begin{thm}\label{th:anyf}
Let $X\sim E(\mu,K^{-1})$. Then $K_{ij}=0$ if and only if $\cov(g(X_i),X_j|X_{(ij)})=0$ for \emph{any}  function $g$ for which this covariance exists. 	
%Equivalently, for every strictly monotone $g$ the regression function $\E(X_j|g(X_i),X_{(ij)})$ depends only on $X_{(ij)}$.
\end{thm}
\begin{proof}
%\footnote{DR. Silly one, in Thm 3.4 we can also state the slightly more general result for $\cov(g(X_I),X_j \mid X_{(ij)})$ for $g:\mathbb{R}^{|I|} \rightarrow \mathbb{R}$, right?}
{The matrix $\Sigma_{ij,ij|(ij)}=\Sigma_{ij,ij}-\Sigma_{ij,(ij)}\Sigma_{(ij),(ij)}^{-1}\Sigma_{(ij),ij}$ is diagonal if and only if $K_{ij}=0$}. The equivalence now follows by applying Lemma~\ref{lem:anyf} to the conditional distribution of $(X_i,X_j)$ given $X_{(ij)}$. {By  Proposition~\ref{prop:margcond}(ii) this distribution is equal to the elliptical distribution $E(\mu_{ij|(ij)},\Sigma_{ij,ij|(ij)})$}.	 
%For the second equivalence we use the fact that $\E(X_j|g(X_i),X_{(ij)})=\E(X_j|X_i,X_{(ij)})$ because $g$ is one-to-one.
\end{proof}

Lemma~\ref{lem:anyf} and Theorem~\ref{th:anyf} are if and only if statements, that is, they characterize the presence of zero marginal and partial correlations (respectively). In particular, Theorem~\ref{th:anyf} characterizes the meaning of elliptical PGs: if $(X_i,X_j)$ are conditionally uncorrelated then so are $X_j$ and any function of $X_i$.
For instance, there is no linear association between $X_j$ and higher-order moments associated to $X_i^2$, $X_i^3$, etc. 

{The most commonly used elliptical distributions in applications of partial correlation graphs are scale mixture of normals. Our interest in general elliptical distributions is motivated by semi-parametric techniques related to {copula models given by} transelliptical distributions, which we now discuss in more detail.}

\subsection{Transelliptical distributions}

Recall that {$Y$} has a transelliptical distribution, denoted $Y \sim \TE(\mu, \Sigma)$, if and only if $f(Y)=(f_1(Y_1),\ldots,f_d(Y_d)) \sim E(\mu,\Sigma)$ for strictly increasing {deterministic functions} $f_i$. If $f(Y)$ is Gaussian (nonparanormal sub-family) the PG gives conditional independence on $Y$  and so it is highly interpretable \citep{liu2012high}.
More generally, a missing edge in the PG means that $\cov(f_i(Y_i),f_j(Y_j)| f_{(ij)}(Y))=0$, but this interpretation is not very interesting given that $f$ is unknown and simply refers to linear {conditional } independence between the latent $(f_i(Y_i),f_j(Y_j))$. The focus should be directly on the dependence structure of $Y$.

Our second main result shows that a weaker version of Theorem~\ref{th:anyf} holds for transelliptical distributions.
\begin{thm}\label{th:somef}
	Suppose $Y \sim \TE(\mu,K^{-1})$. Then $K_{ij}=0$ if and only if the conditional covariance $\cov(f_i(Y_i),g(Y_j)| Y_{(ij)})$ is zero for every function  $g$ for which the covariance exists. 
%	{Equivalently, $\E(f_{i}(X_i)|X_{(i)})$ does not depend on $X_{j}$.}
\end{thm}
\begin{proof}
	Let $X= f(Y) \sim E(\mu,K^{-1})$. Suppose that $\cov(f_i(Y_i),g(Y_j)| Y_{(ij)})=0$ for all $g$, taking $g=f_j$ gives $\cov(f_i(Y_i),f_j(Y_j) \mid Y_{(ij)})=\cov(X_i,X_j \mid Y_{(ij)})=0$. {Since each $f_i$ is a strictly  monotone {fixed function}, it is bijective. It follows that $X_{(ij)}$ and $Y_{(ij)}$ generate the same $\sigma$-field giving that $\cov(X_i,X_j \mid Y_{(ij)})=\cov(X_i,X_j \mid X_{(ij)})$. Vanishing of this conditional covariance is equivalent to $K_{ij}=0$.} To prove the reverse implication, note that 
\begin{align}
  \cov(f_i(Y_i),g(Y_j)| Y_{(ij)})= \cov(X_i, g( f_j^{-1}( X_j)) | X_{(ij)})= 0
  \nonumber
\end{align}
where the second equality follows from Theorem~\ref{th:anyf}. 
%{The last statement follows by a similar argument: $\E(Y_i|X_{(i)})=\E(Y_i|Y_{(i)})$ and the latter does not depend on $Y_j$ by Theorem~\ref{th:anyf} applied to the elliptical vector $Y$ and $h$ the identity function.}
\end{proof}

Theorem~\ref{th:somef} helps interpret the PG as follows. If $K_{ij}=0$ then $f_i(Y_i)$ is conditionally uncorrelated with any function of $Y_j$. Hence learning a single element $f_i$ within $f$ (rather than the whole $f$) describes (local) aspects of conditional dependence of $Y_i$ on $Y_{(i)}$ (and functions thereof). Taking $g$ to be the identity function in Theorem~\ref{th:somef} we get the following result.
\begin{cor}\label{cor:somef}
Suppose $Y \sim \TE(\mu,K^{-1})$. If $K_{ij}=0$ then $\cov(g(Y_i),Y_j| Y_{(ij)})=0$ for \emph{some} strictly increasing function $g$. 
\end{cor}
The function $g$ in this corollary is precisely the function $f_i$ in Theorem~\ref{th:somef}. 
Corollary~\ref{cor:somef} gives the following interpretation. If $K_{ij}=0$ then $Y_i$ is conditionally uncorrelated with some strictly increasing transformation of $Y_j$ and also $Y_j$ is conditionally uncorrelated with some strictly increasing transformation of $Y_i$.

\subsection{Rank correlations}
\label{ssec:rankcorrel}

{Theorem~\ref{th:somef} is an if and only if statement, hence it fully characterizes the meaning of PG in transelliptical distributions}
%Theorem~\ref{th:somef} characterizes PGs 
using covariances between any function of $Y_j$ and latent $f_i(Y_i)$. Kendall's tau gives an interesting alternative characterization that can be interpreted without any reference to $f$.

Let {$W=(W_1,\ldots,W_d)$} be a continuous random vector and {$W'=(W_1',\ldots,W_d')$} be an independent copy. Kendall's tau for $(W_i,W_j)$ is 
	$$
\tau({W_i,W_j})\;:=\;{\rm corr}\big({\rm sign}{(W_i - W_i')},{\rm sign}{(W_j - W_j')}\big).
	$$
In {elliptically-distributed $X \sim E(\mu,\Sigma)$} the following beautiful result relates Pearson correlations $\rho(X_i,X_j)= {\rm corr}(X_i,X_j)= \Sigma_{ij}/\sqrt{\Sigma_{ii} \Sigma_{jj}}$ with Kendall's tau.
\begin{lem}[\cite{lindskog2003kendall}]\label{lem:tau}
	If $X\sim E(\mu,\Sigma)$ then 
	$$
	\tau(X_i,X_j)\;=\;\frac{2}{\pi}\arcsin(\rho(X_i,X_j)).
	$$ 
      \end{lem}
Let $Y \sim \TE(\mu, \Sigma)$, so that $X=f(Y) \sim E(\mu, \Sigma)$. Since Kendall's tau is invariant under strictly increasing transformations,
$$
\tau(Y_i,Y_j)\;=\;\tau(X_i,X_j)\;=\;\frac{2}{\pi}\arcsin(\rho(X_i,X_j)).
$$
Thus $\rho(X_i,X_j)\;=\;\sin(\tfrac{\pi}{2}\tau(Y_i,Y_j))$ and one can learn the correlation matrix {of $X=f(Y)$ from the observations of $Y$} without learning $f$ {simply by estimating $\tau(Y_i,Y_j)$ and using the above formula that deterministically relates $\rho(X_i,X_j)$ with $\tau(Y_i,Y_j)$}.

Below is another basic corollary of Lemma~\ref{lem:tau} {and the  fact that Kendall's tau  is invariant under monotone transformations}.
\begin{prop}
	If $Y\sim {\rm TE}(0,\Sigma)$ then $\Sigma_{ij}=0$ if and only if $\tau(g(Y_i),h(Y_j))=0$ for all strictly increasing functions $g,h:\R\to \R$. Moreover, $\Sigma_{ij}\geq 0$ if and only if $\tau(g(Y_i),h(Y_j))\geq 0$ for all strictly increasing $g,h:\R\to \R$.
\end{prop}

Define conditional Kendall's correlation as
\begin{equation}\label{eq:kendcond}
\tau(X_i, X_j \mid X_{(ij)})= {\rm corr}(\mbox{sign}(X_i - X_i') ,\mbox{sign}(X_j - X_j') \mid X_{(ij)}),	
\end{equation}
where $(X_i',X_j')$ is an independent copy of $(X_i,X_j)$ from the conditional distribution given $X_{(ij)}$. If $X\sim E(\mu,K^{-1})$ then, by Lemma~\ref{lem:tau} applied to the conditional distribution of $(X_i,X_j)$ given $X_{(ij)}$, we have that
\begin{equation}\label{eq:condt}
\tau(X_i, X_j \mid X_{(ij)})\;=\;\frac{2}{\pi}\arcsin\left(-\frac{K_{ij}}{\sqrt{K_{ii}K_{jj}}}\right).	
\end{equation}
%This suggests an obvious plug-in estimator for conditional Kendall's correlations.

\begin{prop}
	Let $Y \sim \TE(0, K^{-1})$. Then $K_{ij}=0$ if and only if $\tau(g(Y_i), h(Y_j) \mid Y_{(ij)})= 0$ for all strictly increasing $g,h: \mathbb{R} \rightarrow \mathbb{R}$, or equivalently, $\tau(Y_i, Y_j \mid Y_{(ij)})= 0$.
\label{prop:partialcorr_condKendall}
\end{prop}
\begin{proof}
	The last equivalence follows from invariance of $\tau$ under strictly monotone transformations. Let $X= f(Y) \sim E(\mu, \Sigma)$. 
Suppose $\tau(g(Y_i), h(Y_j) \mid Y_{(ij)})=0$ for all strictly increasing $g,h$, then
$0=\tau(Y_i, Y_j \mid Y_{(ij)})= \tau(X_i, X_j \mid Y_{(ij)})$,
since Kendall's tau is invariant to monotone transformations, which by Lemma~\ref{lem:tau} applied to the conditional distribution of $(X_i,X_j)$ given $X_{(ij)}$ implies that $\rho_{ij\cdot V\setminus \{i,j\}}=0$ (c.f. Proposition~\ref{prop:margcond}), or equivalently, $K_{ij}=0$.
To prove the reverse implication, suppose that $K_{ij}=0$, then, by Lemma~\ref{lem:tau}, $\tau(X_i, X_j \mid Y_{(ij)})=0$
and hence $\tau(g(Y_i), h(Y_j) \mid Y_{(ij)})=0$ for all strictly increasing $g,h$.
\end{proof}

\section{Positive dependence in elliptical distributions}\label{sec:TP}

In this section we study PGs in elliptical distributions when one imposes positive dependence.  We begin by recalling two important notions of multivariate positive dependence. We show that {neither} of them is meaningful to learn structure in elliptical PGs. This leads to relaxations given by elliptical distributions whose partial correlations are all nonnegative, which we refer to as positive partial correlation graph (PPG). We then complement the interpretation of PPGs offered by the characterizations in Section~\ref{sec:graphical} by studying positive dependence properties embedded within PPGs.

\subsection{Positive dependence}\label{sec:posdep}

Let {$W$} be a $d$-variate continuous random vector with density function $f$. 
%\begin{defn}
%$W$ is associated (A) if and only if $\cov(g(W),h(W))\geq 0$ for any two coordinate non-decreasing functions $g,h:\R^d\to \R$ for which this covariance exists.
% \end{defn}

\begin{defn}
A random vector {$W$} (or its density function {$p$}) is multivariate totally positive of order two ($\mtp$) if and only if
\begin{equation}\label{eq:mtp2}
	{p(w) p(w')}\;\leq\; {p(\min(w,w'))p(\max(w,w'))} \qquad\mbox{for all } w,w' \in \R^d,
\end{equation}
where {$\min(w,w')=(\min(w_1,w_1'),\ldots, \min(w_d,w_d'))$} is the coordinatewise minimum and {$\max(w,w')=(\max(w_1,w_1'),\ldots, \max(w_d,w_d'))$} the coordinatewise maximum of {$w$ and $w'$}. 
\end{defn}

\begin{defn}  ${W}$ is \emph{conditionally increasing} (CI) if for \emph{every} $i\in V$ and $C\subseteq V\setminus \{i\}$ the conditional expectation $\E(h({W}_i)|{W}_C)$ is an increasing function of ${W}_C$ for every increasing function $h:\R\to \R$. 
\end{defn}

 Good general references are \cite{karlin1980classes} for $\mtp$, \cite{muller2001stochastic} for CI.

\begin{prop}\label{prop:mtpbasic}
	If ${W}$ is $\mtp$/CI then each marginal distribution is $\mtp$/CI. If ${W}$ is $\mtp$/CI then each conditional distribution is $\mtp$/CI.
\end{prop}
The proof for the marginal distribution for CI follows from the definition. For the $\mtp$ property it relies on smart combinatorial arguments; see \cite{karlin1980classes}. The statement for conditional distributions follows from the definitions. 

Another well-known result is that these positivity notions are strictly related, and closed under monotone transforms; see Theorem~3.3 and Proposition~3.5 in \cite{muller2001stochastic} as well as Proposition 3.1 in \cite{fallat:etal:17}.
\begin{thm}\label{th:mtpcia} If a random vector is $\mtp$ then it is conditionally increasing.
%	$$
%	\mtp \quad \Longrightarrow\quad CI \quad \Longrightarrow\quad A
%	$$
\end{thm}

\begin{prop}\label{prop:mtpcia_monotonetransforms}
{Let $g(W)=(g_1(W_1),\ldots,g_d(W_d))$ where $g_1,\ldots,g_d$ are strictly increasing.
Then $W$ is CI $\Leftrightarrow g(W)$ is CI.
Also $W$ is $\mtp$ $\Leftrightarrow g(W)$ is $\mtp$.
}
%Let $Y=(Y_1,\ldots,Y_d)$ and $X=(X_1,\ldots,X_d)=(f_1(Y_1),\ldots,f_d(Y_d))$ where $f_i$ are strictly monotone.
%%Then $Y$ is A $\Leftrightarrow X$ is A.
%Then $Y$ is CI $\Leftrightarrow X$ is CI.
%Also $Y$ is $\mtp$ $\Leftrightarrow X$ is $\mtp$.
\end{prop}

In the Gaussian case both condition (\ref{eq:mtp2}) and CI simplify to an explicit constraint on the inverse covariance $K$. %, whereas association is a strictly weaker property.
A symmetric positive definite matrix $K$ is called an M-matrix if $K_{ij}\leq 0$ for all $i\neq j$. Denote the set of inverses of M-matrices by $\im$. Directly from (\ref{eq:partialcor}), $\Sigma\in \im$ if and only if all partial correlations $\rho_{ij\cdot V\setminus\{i,j\}}$ are nonnegative.
\begin{prop}[Proposition~3.6 in \cite{muller2001stochastic}]\label{prop:gaussequiv}
	Suppose $X$ is a Gaussian vector with covariance $\Sigma$ then
$$
\Sigma\in \im\quad\Longleftrightarrow\quad X \mbox{ is }\mtp \quad\Longleftrightarrow\quad X\mbox{ is }CI.
$$	
\end{prop}

%\begin{prop}[\cite{pitt1982positively}]\label{prop:pitt}
%		Suppose $X$ is a zero-mean Gaussian vector with covariance $\Sigma$ then
%		$$
%		X\mbox{ is associated }\quad\Longleftrightarrow\quad \Sigma_{ij}\geq 0\;\; \mbox{ for all }\;\;i,j.
%		$$
% \end{prop}

\subsection{Positive elliptical distributions}\label{sec;impossible}

We first show in Theorem~\ref{th:notCI} that the positive dependence notions reviewed in Section~\ref{sec;impossible} are not useful in our {elliptical} setting. If $K$ has any zeroes then $X$ cannot be CI (hence neither $\mtp$, from Theorem~\ref{th:mtpcia}) {unless $X$ is Gaussian}.
The same impossibility result applies to transelliptical families (outside the nonparanormal sub-family). As a consequence, it is not possible to learn structure (remove edges) of a non-normal elliptical graphical model under these positivity constraints.
%\footnote{DR. Maybe worth emphasizing Theorem \ref{th:mtpcia} in our reply to the AE? For purposes of structural learning, this result is pretty much definitive. }

Even if one were to forsake structural learning and focus on the fully dense graph with no missing edges, it is not possible to find $\mtp$/CI transelliptical distributions, except in very restrictive cases. For example, Proposition \ref{prop:notCI} shows that there are no $\mtp$/CI t-distributions. We defer a deeper analysis to Section~\ref{sec:abdous}, where we fully characterize the elliptical $\mtp$ class and show that it is highly restrictive, particularly as $d$ grows.

We conclude the current section by defining positive transelliptical distributions to be those for which $\rho_{ij\cdot V\setminus \{i,j\}}\geq 0$ for all $i,j\in V$  (equivalently, $K$ being an M-matrix, following upon \cite{agrawal2019covariance}) and showing basic properties such as closedness under margins, conditionals and increasing transforms. We also give properties important for inference, such as positivity of partial correlations given any conditioning set, {partial} faithfulness and that Simpson's paradox cannot occur.

\begin{rem}\label{rem:positivePC}
{Suppose $X\sim E(\mu,K^{-1})$. }From (\ref{eq:meaninK}) it follows that $K$ is an M-matrix if and only if for every $i\in V$ the conditional expectation $\E(X_i|X_{(i)})$ is increasing in $X_{(i)}$. Note that, if $X$ is CI then $\E(X_i|X_{(i)})$ must be increasing in $X_{(i)}$ and so, in particular, $\rho_{ij\cdot V\setminus \{i,j\}}\geq 0$ for all $i,j\in V$. 	 This shows that nonnegativity of all partial correlations is a necessary condition for $X$ to be CI and so also for $X$ to be $\mtp$.
\end{rem}

\begin{thm}\label{th:notCI}
	Suppose that $X\sim E(\mu,K^{-1})$ and $X$ is CI. If $K$ has a zero entry then $X$ is Gaussian. 	Further, suppose that $Y \sim \TE(\mu,K^{-1})$ and $Y$ is CI. Let $X=f(Y) \sim E(\mu,K^{-1})$, if $K$ has a zero entry then $X$ is Gaussian.
\end{thm}
\begin{proof}
	Let $X \sim E(\mu,K^{-1})$ be CI and suppose $K_{ij}=0$. By Proposition~\ref{prop:condcorr}, the conditional covariance $\cov(X_i,X_j|X_{(ij)})$ is zero. Since $X$ is CI, by Proposition~\ref{prop:mtpbasic}, the conditional distribution of $(X_i,X_j)$ given $X_{(ij)}$ is also CI. It is well known that CI distributions are also associated; c.f. \cite{colangelo2005some}.  By Corollary 3 in  \cite{newman1984asymptotic}  applied to this conditional distribution we get that $\cov(X_i,X_j|X_{(ij)})=0$ implies $X_i\indep X_j|X_{(ij)}$. From Proposition~\ref{prop:nocondind} we know that the latter is only possible if $X$ is Gaussian. Consider now $Y \sim \TE(\mu,K^{-1})$. From Proposition~\ref{prop:mtpcia_monotonetransforms}, $X=f(Y)$ is CI and, since $X$ is elliptical, by the first  part of the proof, $X$ must be Gaussian.
\end{proof}

Zeros in the inverse covariance matrix are not the only obstacle for the CI property.

\begin{prop}\label{prop:notCI}
	If $X$ has a multivariate t-distribution then $X$ is not CI.
\end{prop}
\begin{proof}
Since {both} the CI property {and $t$-distribution are} closed under taking margins, it is enough to show that no bivariate t-distribution is conditionally increasing. Suppose $(X_1,X_2)$ has bivariate t-distribution with $k$ degrees of freedom. Without loss of generality assume that the mean is zero and that the scale matrix $\Sigma$ satisfies $\Sigma_{11}=\Sigma_{22}=1$, $\Sigma_{12}=\rho$. By Remark~\ref{rem:positivePC}, necessarily $\rho\geq 0$. Moreover, if $\rho=0$ the statement follows from Theorem~\ref{th:notCI} so assume $\rho>0$. The conditional distribution of $X_1$ given $X_2=x_2$ is a t-Student distribution with $k^*=k+1$ degrees of freedom, $\mu^*=\rho x_2$, and scale parameter $$\sigma^*=\sqrt{\frac{1-\rho^2}{k+1}(k+x_2^2)}$$
	(c.f. Section~5 in \cite{roth2012multivariate}). {To show that $(X_1,X_2)$ is not CI we provide an increasing function $h$ for which  $\E(h(X_1)|X_2=x_2)$ is not increasing in $x_2$.} Let $h(x_1)=\indic_{[k,+\infty)}(x_1)$, {so that} $\E(h(X_1)|X_2=x_2)=1-F_{X_1|X_2}(k|x_2)$, where $F_{X_1|X_2}$ is the 
{conditional cumulative distribution function (c.d.f.).}
%c.d.f. of the conditional distribution. 
Using the formula \cite[(28.4a)]{johnson1994continuous} for the c.d.f. of the t-Student distribution, if $\mu^*\leq k$ (or equiv. $x_2\leq k/\rho$), we express $\E(h(X_1)|X_2=x_2)$ in terms of the incomplete beta function  
	$$
	\E(h(X_1)|X_2=x_2)\;=\;\frac{1}{2}\cdot I_{\alpha(x_2)}\!\left(\frac{k+1}{2},\frac{1}{2}\right)\qquad\mbox{for }\quad  x_2\leq \frac{k}{\rho},
	$$ 
	where
	$$
	\alpha(x_2)\;=\;\frac{k+1}{\left(\tfrac{k-\mu^*}{\sigma^*}\right)^2+k+1}\;=\;\frac{1}{1+\tfrac{(k-\rho x_2)^2}{(k+x_2^2)(1-\rho^2)}}\;\in\;(0,1). 
	$$
	Using the definition of the incomplete beta function in terms of the beta function we get that for a positive constant $C$ and $x_2\leq k/\rho$
	$$
	\E(h(X_1)|X_2=x_2)=C \int_{0}^{\alpha(x_2)} t^{(k-1)/2}(1-t)^{-1/2}{\rm d}t.
	$$
	Since the integral above is strictly increasing in $\alpha(x_2)$, to show that $\E(h(X_1)|X_2=x_2)$ is not increasing, it is enough to show that $\alpha(x_2)$ is not an increasing function for $x_2\leq k/\rho$. But direct calculations show 	$$
	\alpha'(-\rho)=0,\qquad \alpha''(-\rho)=\frac{2k(1-\rho^2)}{(k+1)^2(k+\rho^2)}>0
	$$
	showing that $\alpha$ is strictly decreasing for all $x_2\leq -\rho$ in some neighborhood of $-\rho$.
	\end{proof}

	Proposition~3.3, \cite{ruschendorf2017var} states that for an elliptical distribution $\Sigma\in \im$ if and only if $X$ is CI. Unfortunately, this result is not true as illustrated both by Theorem~\ref{th:notCI} and Proposition~\ref{prop:notCI}. 

%\subsection{Positive elliptical distributions}

Our results show that the CI/$\mtp$ properties are too restrictive in connection with PGs. As a natural alternative, we study the following relaxation proposed by \cite{agrawal2019covariance}.
\begin{defn}
An elliptically  distributed $X\sim E(\mu,\Sigma)$ is positive if $\rho_{ij\cdot V\setminus \{i,j\}}\geq 0$ for all $i,j\in V$ (equiv. $\Sigma\in\im$).	 A transelliptically distributed $Y\sim\TE(\mu,\Sigma)$ is positive if the distribution of $f(Y)\sim E(\mu,\Sigma)$ is positive elliptical. 
%We write $E_+$ and $\TE_+$ to denote these families of distributions.
\end{defn}

{
From Remark \ref{rem:positivePC} this notion of positivity is weaker than both $\mtp$ and CI. Similarly, positivity does not imply association: for example if $K$ is an M-matrix with a zero entry, then it is positive but not associated in general unless Gaussian. Thus, positive elliptical distributions satisfy a fairly weak notion of positive dependence, that is also simple and useful in applied modelling.
}

We first collect basic properties of this family of distributions.
\begin{prop}\label{prop:basicprop}
If $X$ has {a} positive (trans)elliptical distribution then the same is true for each marginal and each conditional distribution.  Positive transelliptical distributions are also closed under strictly increasing transformations. 
\end{prop}
\begin{proof}
	If $X\sim E(\mu,\Sigma)$ then, by Proposition~\ref{prop:margcond}, for every $I\subset V$, $X_I\sim E(\mu_I,\Sigma_{II})$. 	If $\Sigma\in \im$ then  $\Sigma_{II}\in \im$ by \cite{johnson2011inverse}, Corollary 2.3.2. Similarly, $\Sigma\in \im$ then  $\Sigma_{II}-\Sigma_{IJ}\Sigma^{-1}_{JJ}\Sigma_{JI}\in \im$ by \cite{johnson2011inverse}, Corollary 2.3.1, proving that the conditional distribution of $X_I$ given $X_J$ is a positive elliptical distribution. The same argument after replacing $X$ with $f(Y)$ works for transelliptical distributions. The last statement follows directly from the definition of transelliptical distributions.
\end{proof}	
	
	The next proposition shows that positive elliptical distributions retain some strong properties of $\mtp$ Gaussian distributions. 

\begin{prop}\label{prop:tpell}
If $X$ has a positive elliptical distribution then for all $i\in V$ and $C\subseteq V\setminus \{i\}$ the conditional mean $\E(X_i|X_C)$ is an increasing function of $X_C$. Moreover, for any two $i,j\in V$ and $C\subseteq V\setminus \{i,j\}$ it holds that 
$${\rm corr}(X_i,X_j|X_C)\geq 0$$
and
$$
{\rm corr}(X_i,X_j|X_C)=0\quad\Longrightarrow\qquad {\rm corr}(X_i,X_j|X_D)=0\quad \mbox{ for all } D\supseteq C.
$$

\end{prop}
\begin{proof}
These results are well known for Gaussian $\mtp$ distributions; c.f. \cite{fallat:etal:17}. It is convenient to translate them to equivalent statements in terms of $\Sigma$; c.f. \cite[Proposition 3.1.13]{drton2008lectures}. The statement about the conditional mean and the first statement about conditional correlations follow from the fact that $\im$-matrices are closed under taking principal submatrices; c.f. \cite{johnson2011inverse}, Corollary 2.3.2. In consequence, for all $i,j\in V$ and $C\subseteq V\setminus \{i,j\}$ it holds that $(\Sigma_{C\cup\{i,j\},C\cup\{i,j\}})^{-1}_{ij}\leq 0$. The last part states that if $\det\Sigma_{C\cup\{i\},C\cup\{j\}}=0$ for some $C\subseteq V\setminus \{i,j\}$ then $\det\Sigma_{D\cup\{i\},D\cup\{j\}}=0$ for every $D\supseteq C$. This statement is given in  \cite[Theorem~3.3]{johnson2011inverse}.
\end{proof}

These properties are pivotal in the interpretation and application of the classical positive dependence measures. Briefly, the first part says that for positive elliptical distributions conditional correlations are {non-negative}, regardless of what subset of variables one conditions upon. The second part says that if a covariance conditional on $X_C$ is 0, then it remains 0 when conditioning upon larger sets. In particular, zero marginal correlation implies zero partial correlation, hence Simpson's paradox cannot occur.

The following result offers an extension of Theorem~\ref{th:anyf} to the positive case.

\begin{prop}\label{prop:ellip_positive_partialcorr}	If $X \sim E(\mu,\Sigma)$ where $\Sigma \in \im$
then ${\rm corr}(g(X_i),X_j|X_C)\geq 0$ for every $i,j\in V$, any increasing function $g:\R\to \R$, and any conditioning set $C\subseteq V\setminus \{i,j\}$. 
\end{prop}
\begin{proof}
	Let $A=C\cup \{i\}$. By the law of total expectation
	\begin{align}
	\cov(g(X_i),X_j|X_C)\;=\;
{\E_{X_i \mid X_C} \left[ \E \Big((g(X_i)-\E g(X_i)) (X_j-\mu_j) \mid X_i,X_C \Big) \right]}
\nonumber \\
={\E_{X_i \mid X_C} \left[ (g(X_i)-\E g(X_i)) \E \Big( X_j-\mu_j \mid X_A \Big) \right]}
\nonumber \\
={\E_{X_i \mid X_C} \left[ (g(X_i)-\E g(X_i)) \tilde{h}(X_A) \right].}
%\nonumber \\
%=\E_{X_i}\Big((g(X_i)-\E g(X_i))\E(X_j-\mu_j|X_i,X_C)\Big|X_C\Big).
\nonumber
\end{align}
	where we denote $\tilde h(X_A)=\E(X_j-\mu_j|X_A)=\Sigma_{j,A}\Sigma_{A}^{-1}(X_{A}-\mu_A)$. Denote $\tilde K=(\Sigma_{A\cup \{j\}})^{-1}$, which is an M-matrix by \cite{johnson2011inverse}, Corollary 2.3.2. {Using that $\tilde{K}$ is an M-matrix, expression } (\ref{eq:meaninK}) gives that {all entries in the vector $\Sigma_{j,A} \Sigma_A^{-1}$ are non-negative, hence} $\tilde h(X_A)$ is a non-decreasing function of $X_A$. {Further, note that for any fixed} $X_C$ we can write $\tilde h(X_A)$ as a function of $X_i$ only. Denote this function by {$h_{X_C}(X_i)$}. We thus have
	$$
	\cov(g(X_i),X_j|X_C)\;=\;\cov(g(X_i),{h_{X_C}(X_i)}|X_C),
	$$
	where now both $g$ and $h$ are nondecreasing functions of $X_i$. By Property~3 in \cite{esary1967association} it follows that $\cov(g(X_i),{h_{X_C}}(X_i)|X_C)\geq 0$. 

\end{proof}

\begin{rem}
Consider a multivariate function $g(X_{(j)}): \R^{d-1} \rightarrow \R$. Given $X_{(ij)}$ then $g(X_{(j)})$ is an increasing function $\tilde{g}$ of only $X_i$, hence we may write 
$$\cov(g(X_{(j)}),X_j \mid X_{(ij)})=
\cov(\tilde{g}(X_i),X_j \mid X_{(ij)}) \geq 0,
$$
the right-hand side following from Proposition \ref{prop:ellip_positive_partialcorr}.	
\end{rem}

Many constraint-based structure learning algorithms, like the PC algorithm \citep{spirtes2000causation}, rely on the assumption that the dependence structure in the data-generating distribution reflects faithfully the graph. We say that the distribution of $X$ is faithful to a graph $G$ if we have that $X_i\indep X_j|X_C$ if and only if the subset of vertices $C$ separates vertices $i$ and $j$ in $G$, for any $C \subseteq V \setminus \{i,j\}$. In words, any independence obtained by conditioning on subsets $C$ is reflected in the graph. Under faithfulness one can consistently learn the underlying graph from data by conditioning on potentially smaller subsets than the full set of vertices, and benefit from simpler computation. %by learning the underlying conditional independence statements. 
One may extend this definition to partial correlation graphs \citep{spirtes2000causation}: the distribution of $X$ is {\it linearly faithful} to an undirected graph $G$ if we have that ${\rm corr}(X_i,X_j|X_C)=0$ if and only if $C$ separates $i$ and $j$ in $G$.  \cite{buhlmann2010variable} proposed a related convenient notion {unrelated to any particular graph}: the distribution of $X$ is {\it partially faithful} if we have that ${\rm corr}(X_i,X_j|X_C)=0$ for any $C\subset V\setminus \{i,j\}$ implies that ${\rm corr}(X_i,X_j|X_{(ij)})=0$. Using partial faithfulness \cite{buhlmann2010variable} developed a simplified version of the PC algorithm that is computationally feasible even with thousands of variables and was reported to be competitive to standard penalty-based approaches. An important property of positive elliptical distributions is given by the following result.
\begin{prop}\label{th:faitful}
	Every positive elliptical distribution is  partially faithful. 
\end{prop}
\begin{proof}
 The proof  follows from Proposition~\ref{prop:tpell}.
\end{proof}

\subsection{Characterisation of $\mtp$ elliptical distributions}\label{sec:abdous}

We finish our discussion of positive dependence for elliptical distributions with a complete characterization of $\mtp$ distributions. {In this section we assume that $X$ admits a density with respect to the Lebesgue measure. In this case the density necessarily takes the form
\begin{equation}\label{eq:density}
	p(x)\;\;=\;\;|\Sigma|^{-1/2}\,\varphi_d\Big((x-\mu)^T\Sigma^{-1}(x-\mu)\Big),
\end{equation}
where $\varphi_d$ is a nonnegative function (that may depend on $d$) called the density generator. }

Proposition 1.2 in \cite{abdous2005dependence} gives a necessary and sufficient condition for bivariate elliptical distributions to be $\mtp$. With a bit of matrix algebra their proof generalizes.

%Recall from Remark~\ref{rem:density} that the density of $X$ is uniquely given by the density generator $\varphi_d$. {For notational simplicity we drop the dependence on $d$ in the result below. }

\begin{thm}\label{th:abdous}
	Suppose $X$ has a $d$-dimensional elliptical distribution with partial correlations $\rho_{ij \cdot V \setminus \{i,j\}} \geq 0$ {and suppose $X$ admits a twice differentiable density function with the density generator $\varphi_d(t)$ for $t\geq 0$}.  Let $\rho_*={\min_{i,j}} \rho_{ij \cdot V \setminus \{i,j\}}$ and let $\phi(t)=\log \varphi_d(t)$. Then $X$ is $\mtp$ if and only if $\phi'(t)\leq 0$ {for} all $t\geq 0$; $\phi'(t)=0$ implies $\phi''(t)=0$ {for all $t>0$}; and 
\begin{equation}\label{eq:abdous}
-\frac{\rho_{*}}{1+\rho_{*}}\;\;\leq\;\;\frac{ t \phi''(t)}{\phi'(t)}\;\;\leq\;\;\frac{\rho_{*}}{1-\rho_{*}}.	
\end{equation}
	for all $t\in \T=\{t:\phi'(t)< 0\}$. In particular, $\inf_{t\in \T} \tfrac{t\phi''(t)}{\phi'(t)}> -\tfrac{1}{2}$.
\end{thm}
\begin{proof}
Without loss of generality assume $X$ has mean zero and $K=\Sigma^{-1}$ satisfies $K_{11}=\cdots=K_{dd}=1$. In this case $K_{ij}=-\rho_{ij \cdot V \setminus \{i,j\}}$ for all $i\neq j$. If $X$ admits a strictly positive {and twice differentiable} density function $f(x)$ then $X$ is $\mtp$ if and only if for every $1\leq i< j\leq d$
$$
\frac{\partial^2 }{\partial x_i\partial x_j} \log f(x)\;\geq\; 0\qquad \mbox{for all }x\in \R^d.
$$	
This result, found for example in \cite{bach2019submodular} can be proved by elementary means, for example, by applying a second-order mean value theorem (Theorem~9.40 in \cite{rudin1964principles}). In our case $f(x)=|\Sigma|^{-1/2}\varphi_d(x^T K x)$ so $f$ is $\mtp$ if and only if for every $1\leq i< j\leq d$
\begin{equation}\label{eq:mtp2cond}
\frac{\partial^2 }{\partial x_i\partial x_j} \phi(x^T K x)\;\geq\; 0\qquad \mbox{for all }x\in \R^d.	
\end{equation}
Basic calculus gives $\nabla \phi(x^T K x)\;=\;2\,\phi'(x^T K x) K x$ and 
$$
\nabla^2 \phi(x^T K x)\;=\;2\,\phi'(x^T K x)K +4\,\phi''(x^T K x)K x x^T K.
$$
{If $x\neq 0$ we perform a change of coordinates to $t$ and $v=\tfrac{1}{\sqrt{t}}Kx$ constrained by $t>0$ and $v^T\Sigma v=1$. We extend it to the whole $\R^d$ by mapping $x=0$ to $t=0$ and $v=0$. In the new coordinate system, condition (\ref{eq:mtp2cond}) holds if and only  if for all $1\leq i<j\leq d$} 
\begin{equation}\label{ineq:mtp2aux}
2\phi'(t) K_{ij} + 4\phi''(t)\, t \,v_i v_j\;\;\geq \;\; 0\qquad\mbox{for all }t,v.	
\end{equation}

Taking $v$ such that $v_i=0$  {(which includes the $x=0$ case)}, (\ref{ineq:mtp2aux}) implies that $\phi'(t)\leq 0$ for all $t\geq 0$, which gives the first condition. If $\phi'(t)=0$ {and $t>0$}, (\ref{ineq:mtp2aux}) cannot hold for all $v$ unless $\phi''(t)=0$, which gives the second condition in the theorem. Now suppose $t$ is such that $\phi'(t)<0$ then (\ref{ineq:mtp2aux})  becomes
\begin{equation}\label{ineq:mtp2aux2}
2v_iv_j\frac{ t \phi''(t)}{\phi'(t)}\;\;\leq\;\; -K_{ij}.
\end{equation}
To study the bounds on $2v_iv_j$ subject to $v^T\Sigma v=1$ we define the Lagrangian
$$
L(v,\lambda)\;=\;2v_iv_j-\lambda(v^T \Sigma v-1).
$$
Denote $A=\{i,j\}$ and $B=V\setminus A$. The Lagrangian condition {$\nabla_v L(v,\lambda)=0$} can be then reduced to $v_B=-\Sigma^{-1}_{BB}\Sigma_{BA}v_A$ and 
$$
\begin{bmatrix}
	v_j\\
	v_i
\end{bmatrix}\;=\;\lambda(\Sigma_{AA}-\Sigma_{AB}\Sigma_{BB}^{-1}\Sigma_{BA})\begin{bmatrix}
	v_i\\
	v_j
\end{bmatrix}\;=\;\lambda K_{AA}^{-1}\begin{bmatrix}
	v_i\\
	v_j
\end{bmatrix}.
$$
Multiplying both sides by $K_{AA}$ we get
$$
\begin{bmatrix}
	1 & K_{ij} \\
	K_{ij} & 1
\end{bmatrix}\begin{bmatrix}
	v_j\\
	v_i
\end{bmatrix}\;=\;\lambda \begin{bmatrix}
	v_i\\
	v_j
\end{bmatrix}.
$$
All stationary points must then satisfy $v_i^2=v_j^2$. The maximal value of $2v_i v_j$ subject to $v^T\Sigma v=1$ is $2\alpha^2$ obtained at a point where $v_i=v_j=\alpha$. The value of $\alpha$ can be found by noting that 
$$
v^T\Sigma v\;=\;\alpha^2\bs 1^T(\Sigma_{AA}- \Sigma_{AB}\Sigma_{BB}^{-1}\Sigma_{BA})\bs 1 \;=\;\frac{2\alpha^2}{1+K_{ij}},
$$
where $\bs 1$ is the vector of ones. Since $v^T\Sigma v=1$, $2\alpha^2=1-\rho_{ij \cdot V \setminus \{i,j\}}$. In a similar way we show that the minimal value of $2v_iv_j$ is $-(1+\rho_{ij \cdot V \setminus \{i,j\}})$. This gives that (\ref{ineq:mtp2aux2}) is equivalent to
$$
-\frac{\rho_{ij \cdot V \setminus \{i,j\}}}{1+\rho_{ij \cdot V \setminus \{i,j\}}}\;\;\leq\;\;\frac{ t \phi''(t)}{\phi'(t)}\;\;\leq\;\;\frac{\rho_{ij \cdot V \setminus \{i,j\}}}{1-\rho_{ij \cdot V \setminus \{i,j\}}}.
$$
This inequality must be satisfied for every $i\neq j$. However the functions $\rho/(1+\rho)$ and $\rho/(1-\rho)$ are increasing for $\rho\in [0,1)$ and so $\min_{ij}\rho_{ij \cdot V \setminus \{i,j\}}/(1-\rho_{ij \cdot V \setminus \{i,j\}})=\rho_*/(1-\rho_*)$ and $\min_{ij}\rho_{ij \cdot V \setminus \{i,j\}}/(1+\rho_{ij \cdot V \setminus \{i,j\}})=\rho_*/(1+\rho_*)$. Thus we arrive at (\ref{eq:abdous}). 

Now suppose that $\phi$ is such that $\phi'(t)\leq 0$; $\phi'(t)=0$ implies that $\phi''(t)=0$; and    (\ref{eq:abdous}) holds for all $t\in \T$. By reversing the argument above we conclude that (\ref{ineq:mtp2aux}) holds for all $t\in \T$. For all the remaining $t$ this inequality also holds because then both sides are equal to zero. However, as we argued before (\ref{ineq:mtp2aux}) holds for all $t$ if and only if $X$ is $\mtp$. This concludes our proof.
\end{proof}

We illustrate Theorem~\ref{th:abdous} with two examples.
\begin{ex}
By Proposition~\ref{prop:notCI}, if $X$ has {a} t-distribution then $X$ is not conditionally increasing and so, in particular, it is not $\mtp$.	Theorem~\ref{th:abdous} provides an easy way to see this. For the d-dimensional t-distribution with $k$ degrees of freedom $\phi(t)=-\tfrac{k+d}{2}\log(1+\tfrac{t}{k})$. Since $\phi'(t)=-\tfrac{1}{2}\tfrac{k+d}{k+t}<0$, condition~(\ref{eq:abdous}) requires that 
	$$
	-\frac{\rho_*}{1+\rho_*}\leq -\frac{t}{k+t}\leq \frac{\rho_{*}}{1-\rho_{*}}
	$$
	must be satisfied for all $t\in \R$. Taking the limits $t\to\pm\infty$ shows that this is impossible irrespective of $\rho_*\in (-1,1)$. Similarly, in the case of a zero-mean multivariate Laplace distribution the density generator is $\varphi_d(t)=(\tfrac{t}{2})^{k/2} K_k(\sqrt{2t})$, where $k=\tfrac{2-d}{2}$ and $K_k(\cdot)$ is the modified Bessel function of the second kind. Irrespective of $d$, $\tfrac{t\phi''(t)}{\phi'(t)}\in (-1,-\tfrac{1}{2})$ and so {Laplace} distributions are never $\mtp$ {by the last part of Theorem~\ref{th:abdous}}.
\end{ex}

\begin{ex}\label{ex:alpha}
	Consider a d-dimensional $X \sim E(\mu, K^{-1})$ with generator $\varphi_d(t)=e^{-t^\alpha}$, i.e. density $f(x) \propto e^{-(x^T K x)^\alpha}$. This defines a special type of Kotz type distributions;  see \cite{fang2018symmetric} (Section~3.2) and in our case we have
	$$
	\frac{t\phi''(t)}{\phi'(t)}\;=\;\alpha-1.
	$$From Theorem~\ref{th:abdous}, if $\alpha>1$ ($X$ has thinner-than-Normal tails) then $\mtp$ holds if and only if $\rho_* > 1-1/\alpha$, and similarly if $1/2<\alpha<1$ (thicker-than-Normal tails) then $\mtp$ also holds if and only if $\rho_* \geq 1/\alpha-1$. If $\alpha \leq 1/2$ then $X$ cannot be $\mtp$. 
\end{ex}

The constraints on possible $\rho_*$ in Example~\ref{ex:alpha} did not take into account one more important aspect of the problem, namely that $K$ is a $d\times d$ positive definite matrix. To illustrate this, suppose that all off-diagonal entries of $K$ are equal, that is, $\rho_{ij \cdot V \setminus \{i,j\}}=\rho_*=-K_{ij}>0$ for all $i\neq j$. Such $K$ is positive definite if and only if $\rho_*<1/(d-1)$. In Example~\ref{ex:alpha} this gives an upper bound on $\rho_*$ that interplays with the lower bound $\rho_*\geq |1-1/\alpha|$. These two bounds define a non-empty set if and only if $|1-1/\alpha|<1/(d-1)$. If $d=2$ this holds for any $\alpha>1/2$. If $d\geq 3$ this holds if and only if $$1-\frac{1}{d} \;<\; \alpha \;<\; 1+\frac{1}{d-2}.$$
{Note that unless $\alpha=1$ (the Gaussian case), this condition cannot hold for all $d\in \N$.}

It is remarkable that this simple example generalizes and yields the following characterization of elliptical families with a fixed density generator that contain $\mtp$ distributions. 
%Recall from Remark~\ref{rem:density} that a $d$-variate elliptical distribution with density generator $\varphi_d$ admits the density function $f(x)= |\Sigma|^{-1/2} \varphi_d\big((x-\mu)^T\Sigma^{-1}(x-\mu)\big)$ and the density generator in this representation is independent of $d$.

\begin{thm}\label{thm:necessary_mtp2}
Consider the family of all elliptical distributions with density generator $\varphi_d(t)$ and let $\phi(t)=\log\varphi_d(t)$. Then, there exists a scale matrix parameter $\Sigma$ such that the density (\ref{eq:density}) is $\mtp$ if and only if $\phi'(t)\leq 0$; $\phi'(t)=0$ implies $\phi''(t)=0$; and
$$
-\frac{1}{d}\;\;<\;\;\frac{t\phi''(t)}{\phi'(t)}\;\;<\;\;\frac{1}{d-2}
$$
for all $t\in \T=\{t:\phi'(t)< 0\}$. \end{thm}

\begin{proof}
Let $\beta_*=\inf_{t\in\T}\tfrac{t\phi''(t)}{\phi'(t)}$ and $\beta^*=\sup_{t\in T}\tfrac{t\phi''(t)}{\phi'(t)}$. First suppose that for some $K=\Sigma^{-1}$ the underlying elliptical distribution is $\mtp$. By Theorem~\ref{th:abdous} it follows that $\phi'(t)\leq 0$ and $\phi''(t)=0$ whenever $\phi'(t)=0$. Assume without loss that $K$ is normalized to have ones on the diagonal so that 	$K_{ij}=-\rho_{ij\cdot V_{\setminus \{i,j\}}}\leq 0$ for $i\neq j$. The fact that partial correlations must be necessarily nonnegative follows from Remark~\ref{rem:positivePC}. Since $K$ is an M-matrix, by Proposition~6.1 in \cite{lauritzen2019maximum} we can shrink each off-diagonal entry towards zero preserving positive-definitedness. In particular, the matrix $K'=(1+\rho_*)\mathbb I-\rho_*\bs 1\bs 1^T$ obtained from $K$ by replacing each off-diagonal entry with $-\rho_*$, where $\rho_*=\min_{i,j}\rho_{ij\cdot V\setminus \{i,j\}}$, must be positive definite. This matrix is positive definite if and only if $\rho_*<1/(d-1)$, which gives an upper bound on $\rho_*$ on the top of the two upper bounds implied by Theorem~\ref{th:abdous}, namely, $\rho_*\geq \beta^*/(1+\beta^*)$ and $\rho_*\geq -\beta_*/(\beta_*+1)$ {(use the fact that $\beta_*>-1/2$ by the last part of Theorem~\ref{th:abdous})}. The intersection of these three constraints is non-empty if and only if $\tfrac{1}{d-1}> \max\{\tfrac{\beta^*}{1+\beta^*},-\tfrac{\beta_*}{\beta_*+1}\}$. In other words, $\beta^*< 1/(d-2)$ and $\beta_*>-1/d$, which finished the proof of one implication.  

Now suppose $\phi'(t)\leq 0$ and $\phi''(t)=0$ whenever $\phi'(t)=0$. If, in addition, the third condition in the theorem is satisfied then $\beta_*>-1/d$ and $\beta^*<1/(d-2)$. Let $K$ be a matrix with ones on the diagonal and $-\rho_*$ on the remaining entries. If $0\leq \rho_*<1/(d-1)$ then $K$ is positive definite. If $\rho_*\nearrow 1/(d-1)$ then $\rho_*/(1-\rho_*)\nearrow 1/(d-2)$ and $-\rho_*/(1+\rho_*)\searrow -1/d$. Since $\beta_*>-1/d$ and $\beta^*<1/(d-2)$ we can always find $\rho_*$ close enough to $1/(d-1)$ so that 
$$-\frac{\rho_*}{1+\rho_*}\;\leq\;\beta_*\;\leq\;\beta^*\;\leq\; \frac{\rho_*}{1-\rho_*}.$$ 
By Theorem~\ref{th:abdous} the corresponding distribution is $\mtp$.  
\end{proof}

To illustrate this result consider the elliptically symmetric logistic distribution as defined in \cite{fang2018symmetric}, Section~3.5. The density generator satisfies $$\varphi_d(t)=\frac{e^{-t}}{(1+e^{-t})^2},\qquad \qquad \frac{t\phi''(t)}{\phi'(t)}=\frac{2t}{e^t-e^{-t}}\in (0,1].$$
Theorem~\ref{th:abdous} gives that a bivariate logistic distribution is $\mtp$ if and only if $\rho_{12}\geq 1/2$. However, if $d\geq 3$, Theorem~\ref{thm:necessary_mtp2} implies that there are no $\mtp$ distributions of this form. 

{To summarize, our results show that although there are some non-normal elliptical distributions that are $\mtp$, the imposed constraints can be quite severe, particularly as the dimension $d$ grows. We additionally showed that popular elliptical distributions such as the t, Laplace and most Kotz-type distributions in Example \ref{ex:alpha} cannot be $\mtp$. These findings highlight the need to define alternative measures of positive association, such as the PPGs in Section \ref{sec;impossible}.}

\section{Examples}\label{sec:examples}

\begin{figure}
\begin{center}
\begin{tabular}{cc}
\includegraphics[width=0.5\textwidth]{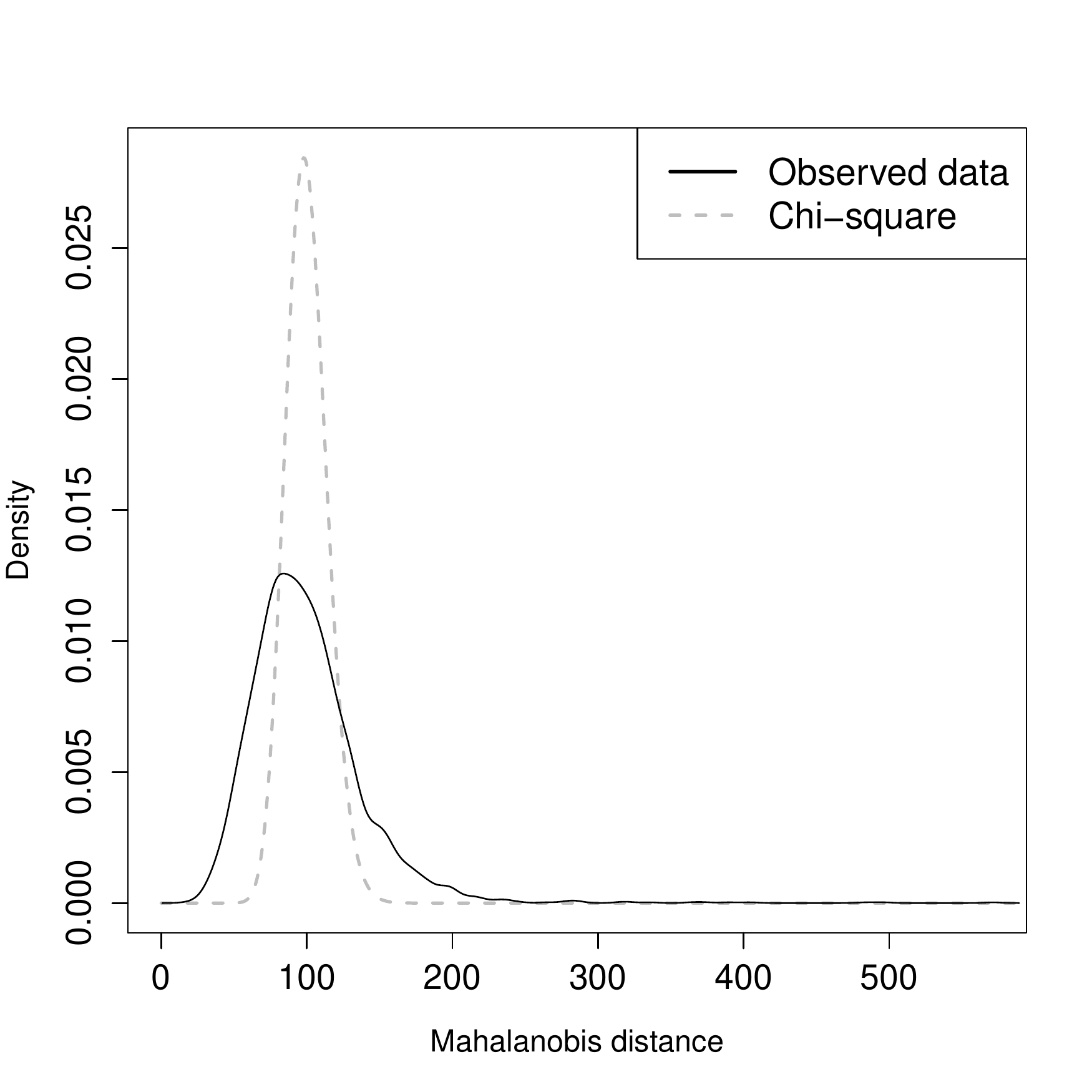} &
\includegraphics[width=0.5\textwidth]{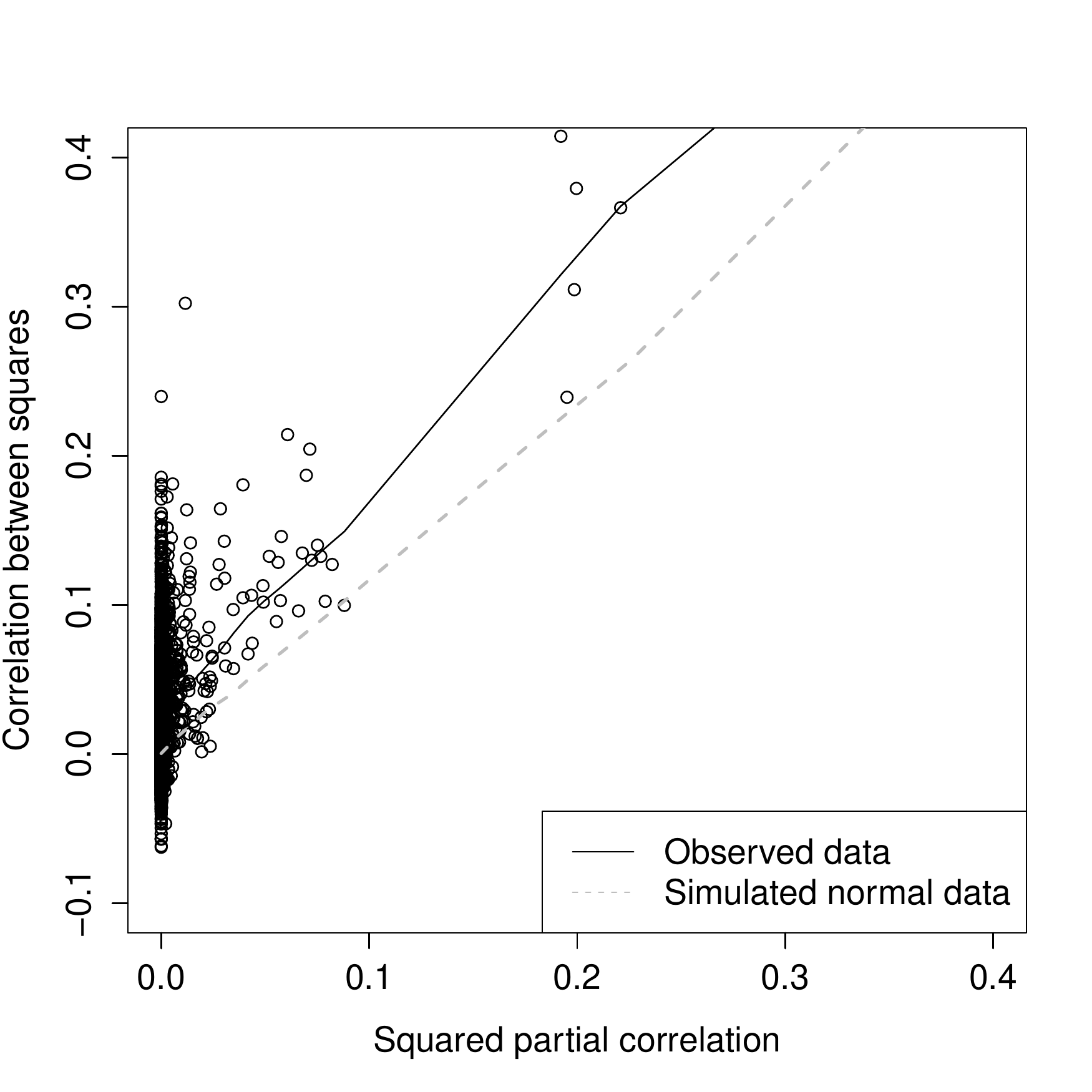} \\
\includegraphics[width=0.5\textwidth]{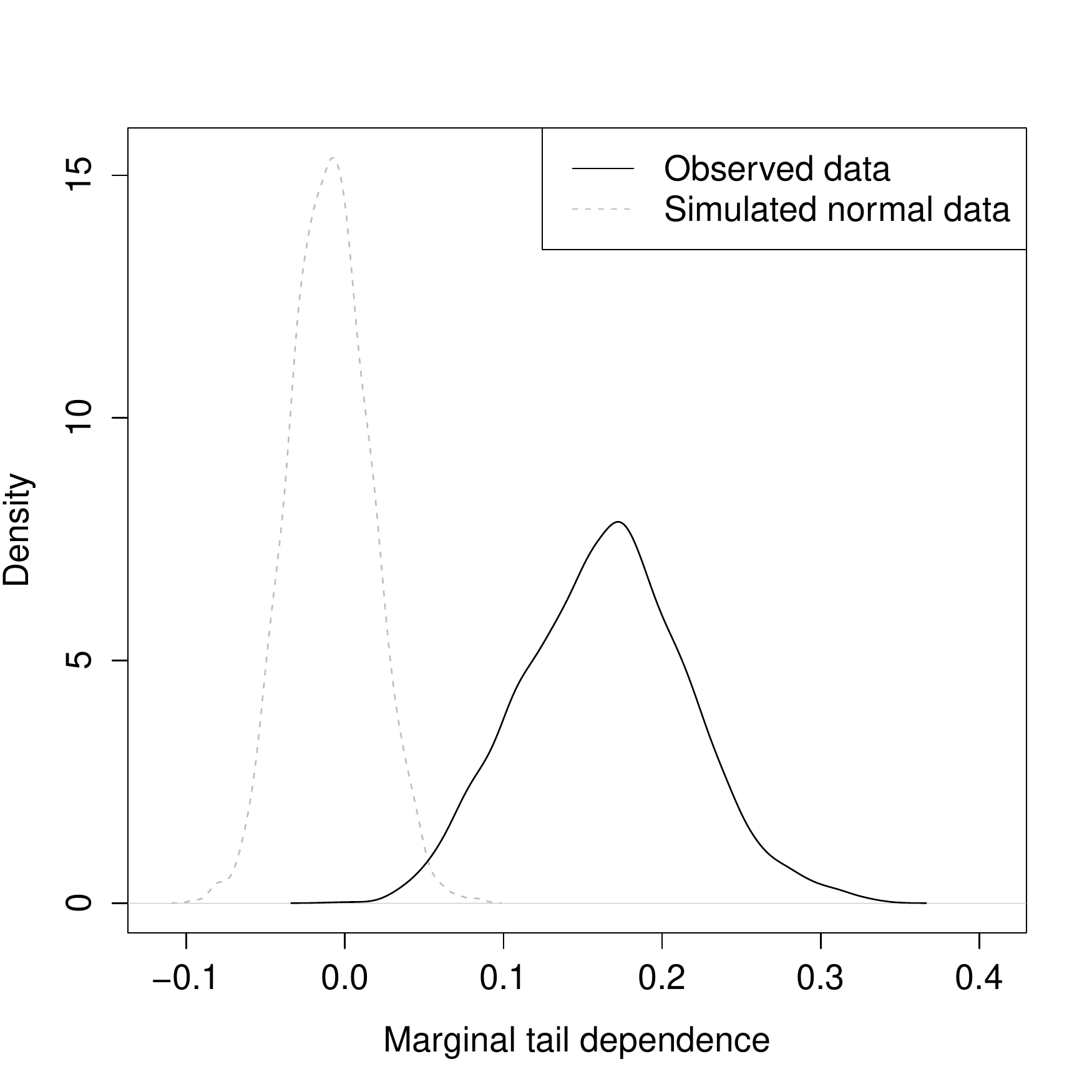} &
\includegraphics[width=0.5\textwidth]{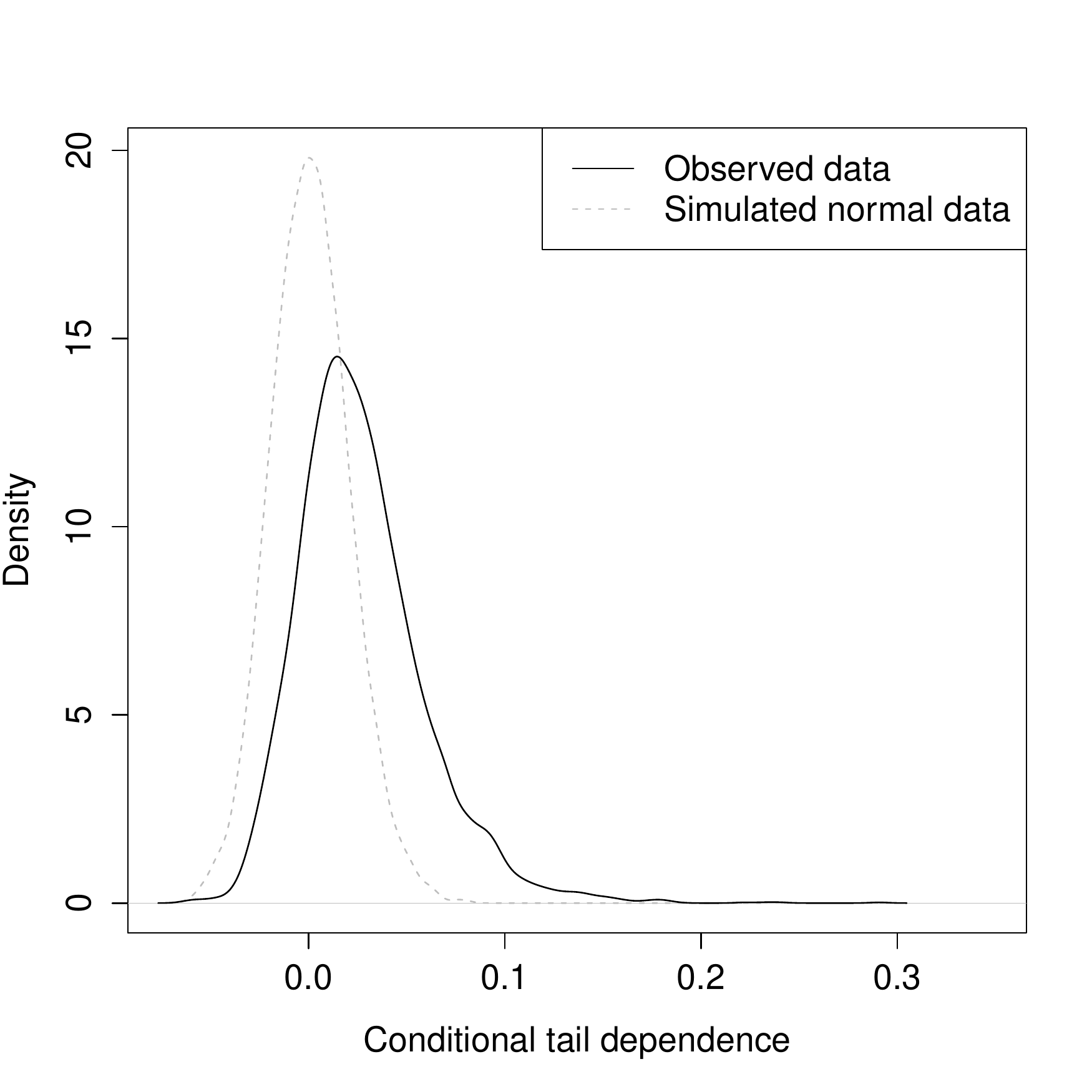} \\
\end{tabular}
\end{center}
\caption{S\&P500 data. Top left: Mahalanobis distances and $\chi^2_{100}$ density.
Top right: conditional tail dependence $\hat{\theta}_{ij \cdot V\setminus \{i,j\}}$ versus normal prediction $\hat{\rho}_{ij \cdot V \setminus \{i,j\}}^2$.
Bottom: distribution of $\hat{\theta}_{ij} - \hat{\rho}_{ij}^2$ and $\hat{\theta}_{ij \cdot V\setminus \{i,j\}} - \hat{\rho}_{ij \cdot V \setminus \{i,j\}}^2$}
\label{fig:sp500}
\end{figure}

We illustrate the application of transelliptical PG and PPG and the interpretation afforded by our characterizations with S\&P500 stock market data. The R code to reproduce our analyses is provided as supplementary material.
We downloaded the daily log-returns of S\&P500 stocks for the 10-year period ranging from 2010-04-29 to 2020-04-14 ($n=2,514$ observations). For illustration we selected the first $d=100$ stocks, hence the graphical model has 4,950 potential edges.
We used the R package huge \citep{zhao2012huge} to apply univariate transformations aimed at improving the marginal normal fit (function \texttt{huge.npn}). 
Despite these transformations, we observed departures from multivariate normality. 
Let the observed and transformed $n \times d$ data matrices be  $Y$ and $X$ (respectively), both with zero column means and unit variances. The empirical distribution of the Mahalanobis distances $(y_{i1},\ldots,y_{id}) S^{-1} (y_{i1},\ldots,y_{id})^T$, where $S$ is the sample covariance, had significantly thicker tails than the $\chi_d^2$ expected under multivariate normal data and $S=\Sigma$ (Figure \ref{fig:sp500}, top left).
{This departure from normality motivates considering other elliptical models.}

We studied the dependence structure in these data via several models.
First we fit a transelliptical model $\mbox{TE}(0,\Sigma)$ to $Y$, where $\Sigma$ is estimated by first computing Kendall's $\tau$ and then exploiting their connection to $\Sigma$ in Lemma \ref{lem:tau}. This procedure can be performed with option \texttt{npn.func = "skeptic"} in function \texttt{huge.npn}, see \cite{liu2012transelliptical} for details. 
Second, we also fit an elliptical model $E(0,\Sigma)$ to $X$.
In both models we estimated $K=\Sigma^{-1}$ via graphical LASSO \citep{friedman2008sparse}, {where the} regularization parameter was set via the EBIC (\cite{chen:2008}, function \texttt{huge.select}), in the transelliptical case using the pseudo-likelihood defined by $\hat{\tau}_{ij}$, see \cite{foygel2010extended}.
The transelliptical model is in principle more robust, in that it does not require estimating the marginal transformations. However both models provided similar results: the Spearman correlation between the estimated $\hat{K}_{ij}$ was 0.911, the selected PGs agreed in 93.0\% of the 4,950 edges, and there were no disagreements in the signs of $\hat{K}_{ij}$ for any $(i,j)$.

To illustrate the interpretation of the PG implied by $\hat{K}$, relative to a Gaussian graphical model, we focus on the elliptical model for $X$.
Figure \ref{fig:sp500} (bottom left) shows that the marginal tail dependence $\hat{\theta}_{ij}$ in \eqref{eq:corrsq} is significantly larger than the $\hat{\rho}_{ij}^2$ expected under normality. The magnitude of these departures is practically significant. For comparison the figure also displays $\hat{\theta}_{ij}$ estimated from simulated Normal data, with zero mean and sample covariance matching that of $X$.

In practical terms, $\theta_{ij}$ measures the predictability of a variable's variance (also called volatility) from that of other variables. A natural question is what predictability remains after conditioning upon other variables, i.e. what is the {\it conditional} tail dependence $\theta_{ij \cdot V \setminus \{i,j\}}$  in \eqref{eq:corrsq_partial}. 
To address this question for each variable pair $(i,j)$ we computed the non-parametric estimate
$$
\hat{\theta}_{ij \cdot V \setminus \{i,j\}}= {\rm corr}(e_i^2,e_j^2 \mid X_{(ij)})
$$
where $e_i= x_i - \hat{\mu}_{i \mid V\setminus \{i,j\}}$, $x_i$ is the $i$-{th} column in $X$ and $\hat{\mu}_{i \mid V\setminus \{i,j\}}$ the least-squares prediction given $X_{(ij)}$ (analogously for $e_j$).
These estimates were significantly larger than the $\hat{\rho}^2_{ij \cdot V\setminus \{i,j\}}$ expected under normality (Figure \ref{fig:sp500}, bottom right).
As a further check, from \eqref{eq:corrsq_partial} the elliptical model {(more specifically, the scale mixture of normals sub-family)} predicts $\theta_{ij \mid V\setminus \{i,j\}}$ to be linear in $\rho^2_{ij \mid V\setminus \{i,j\}}$. Figure \ref{fig:sp500} (top right) suggests that they are indeed roughly linearly related. Admittedly one never expects a model to describe the data perfectly, but the elliptical model appears reasonable to study volatility in these data.

The estimated partial correlation graph had 1,600 out of the 4,950 edges. Our results from Section~\ref{sec:graphical} help strengthen the interpretation of the missing edges, e.g. $\hat{K}_{ij}=0$ suggests that conditional on $X_{(ij)}$ one cannot predict the variance, asymmetry or kurtosis in $x_j$ linearly from $x_i$.
Further it also implies zero Kendall's conditional tau between increasing transforms of $x_i$ and $x_j$, e.g. if daily returns are not conditionally positively/negative correlated (according to Kendall's tau) then neither are log-returns.

A quite interesting point is that among the $1,600$ edges the estimated partial correlations were positive for $1,481$ and negative for only $119$ edges. That is, the partial correlation graph was very close to being a PPG; see \cite{agrawal2019covariance} for a discussion why this may be frequently encountered in stock data, and \cite{epskamp2018tutorial,lauritzen2019total} for examples in Psychology. 
To compare the PPG fit with our earlier graphical LASSO fit we estimated the precision matrix under the constraint that $K$ is an M-matrix, using the R package \texttt{mtp2} available at GitHub \cite[Algorithm 1]{lauritzen2019maximum}\footnote{{Package \texttt{golazo} that recently appeared on GitHub offers a more flexible and scalable way of doing this and other related computations. The relevant function is \texttt{positive.golazo(S,rho=Inf)}, where $S$ is the sample covariance matrix. See \cite{lauritzen2020locally} for more details.}}. The maximized constrained log-likelihood was substantially higher than for the graphical LASSO fit {($-266,361$ versus $-268,773$)} and the graph was sparser ($1,228$ versus $1,600$ edges), hence the EBIC (and any other $L_0$ model selection criteria) strongly favored the PPG model.
Note that, from its Lagrangian interpretation, the graphical LASSO constrains the size $|\hat{K}_{ij}|$. In contrast the M-matrix constraint allows for arbitrarily large $|\hat{K}_{ij}|$, provided $\hat{K}_{ij} \leq 0$. That is, {the graphical LASSO and the PPG} constraints induce quite different regularization and the latter appears more appropriate for these S\&P500 data, illustrating the potential value of positivity constraints in certain applications.

The selected graph being a PPG strengthens its interpretation. By Proposition~\ref{prop:tpell}, the finding suggests that all possible partial correlations are positive regardless of the conditioning set, and that Simpson's paradox does not occur in these data, i.e. stocks with zero marginal correlation also have zero partial correlation. By our earlier discussion, this implies that if $\rho_{ij}=0$  marginally then $x_i$ is uncorrelated with higher moments of $x_j$, both marginally and conditionally on $X_{(ij)}$. Further, the conditional expectation of $x_i$ can only {be} increasing as a function of other variables (or increasing transformations thereof), and missing edges indicate the lack of such association.

\section{Discussion}

When studying multivariate dependence in applications it is often convenient to strike a balance between models that equip strong theoretical properties (e.g. Gaussian, non-paranormal, $\mtp$ and CI classes) but impose potentially restrictive conditions, and models that are more flexible but do not provide such strong characterizations and/or lead to complex interpretations.
We studied a natural strategy based on the transelliptical family and partial correlation graphs, {including many copula models that are popular in applications}. We showed that the interpretation remains simple yet goes far beyond the regular linear dependence. 

%The aim of this paper was to study dependence structures in elliptical and transelliptical distributions. The main motivation was to understand better the interpretation of the partial correlation graphs for (trans)elliptical distributions. We showed that this interpretation goes far beyond the regular linear dependence. We hope that these findings will strengthen the analysis done with (trans)elliptical graphical models. 

This work is also relevant in the context of Gaussian graphical models. Although the partial correlation graph in the Gaussian case translates into conditional independence statements, it is important to understand how robust is this interpretation with respect to the Gaussianity assumption. Our analysis shows that in the elliptical case a lot of this dependence information is retained. 
We also illustrate how simple tail dependence measures, like the one in (\ref{eq:corrsq}), characterize the Gaussian distribution within the {scale mixture of normals} family and can help assess whether transelliptical class is useful to capture second-order dependence (variance) dependencies in the data.

An important part of this paper is the study of positive dependence. The notion of positivity can be quite useful in regularizing inference relative to unrestricted penalized likelihood, as we illustrated in the S\&P500 example. However, we also showed that strictly speaking some standard notions of positive dependence are meaningless for structural learning in elliptical partial correlation graphs. 
One of our main contributions is a remarkable result that characterizes $\mtp$ elliptical distributions and shows that $\mtp$ becomes very restrictive in high dimensions. It is therefore important to study relaxations such as positive elliptical distributions that impose all partial correlations to be nonnegative. We showed that this family retains strong positive dependence properties that are important from the applied point of view. 

In conclusion, we hope that our results help to motivate the study of other suitable relaxations of Gaussianity and positivity in graphical models, as well as strengthen the use of transelliptical graphical models in practice.

%\subsection{Totally positive elliptical distributions}
%
% The fact that the boundary of the set of M-matrices gives zero restrictions on $K$ was a driving idea behind using the $\mtp$ assumption to obtain sparse representations of the data in the Gaussian case \cite{lauritzen2019maximum,pavez2018learning,slawski2015estimation}. Recently \cite{wang2019learning} showed that there is a way of learning consistently the support of $K$ without introducing any regularization. 

%Compare efficiency of the Spearman vs Kendall depending on the tails.
%
%\subsection{No positivity constraints}
%
%THREE APPROACHES - we can rely on transelliptical (loosing efficiency in estimating the correlation matrix) SKEPTIC, follow the standard Gaussian procedure with better rates of leanring the correlation matrix (some guarantees?), 
%
%elliptical 
%
%
% Learning the partial correlation graph $G(K)$ from the data can be done by incorporating a graphical lasso penalty exactly in the same way as it was proposed for the t-distribution in \cite[Section~2]{finegold2011robust}. We will write more after studying more carefully \cite{barber2018rocket}.
%
%
%Here we can first briefly overview what was done in \cite{barber2018rocket} for transelliptical distributions and in \cite{bhadra2018inferring} for scale mixtures. Then we can mention nonparanormal etc In the end the main point may be doing total positivity.
%
%\subsection{Scale mixture of normals}
%
%We could learn nonparametically the distribution of $\tau$ and use form of the EM algorithm. 
%
%\subsection{Learning under positivity}
%
%
%
%\appendix

\section*{Acknowledgements}

We thank Yuhao Wang, Frank R\"{o}ttger, and Ludger R\"{u}schendorf for helpful remarks. {We also thank  the reviewer and the AE of  the first version of this manuscript for pointing to us a mistake in the original proof of Lemma~\ref{lem:anyf} and other inaccuracies}.

\bibliography{bib_mtp2}
\bibliographystyle{authordate1}

\end{document}